\documentclass[12pt,reqno]{amsart}
\usepackage{mathrsfs}
\usepackage{upgreek}
\usepackage{amscd}
\usepackage{amssymb}
\usepackage{amsmath, amsthm, graphicx}
\usepackage{enumitem}
\usepackage[all]{xy}
\usepackage{cite}
\usepackage{tikz}
\usepackage{color}
\usepackage{tikz}
\usepackage{tikz-cd}
\usepackage{hyperref}
\usepackage[english]{babel}
\usepackage[margin=1.0in]{geometry}

\DeclareMathOperator{\Aut}{Aut}
\DeclareMathOperator{\Cont}{Cont}
\DeclareMathOperator{\Coeff}{Coeff}

\newcommand{\Gammaloop}{\Gamma^\mathrm{loop}}

\theoremstyle{plain}
\newtheorem{theorem}{Theorem}
\newtheorem{lemma}[theorem]{Lemma}
\newtheorem{proposition}[theorem]{Proposition}
\newtheorem{corollary}[theorem]{Corollary}
\newtheorem{example}[theorem]{Example}

\newtheorem{conjecture}[theorem]{Conjecture}

\theoremstyle{definition}

\theoremstyle{remark}
\newtheorem{remark}[theorem]{Remark}

\usepackage{color}

\begin{document}
\title{Structures in topological recursion relations}

\author[F.~Janda]{Felix Janda}
\address{Department of Mathematics, University of Illinois Urbana--Champaign, Urbana, IL 61801, USA}
\email{fjanda@illinois.edu}

\author[X.~Wang]{Xin Wang}
\address{School of Mathematics, Shandong University, Jinan, Shandong 250100, China}
\email{wangxin2015@sdu.edu.cn}

\begin{abstract}

  In this paper, we study the basic structures of degree-$g$ topological recursion relations on the moduli space of curves $\overline{\mathcal{M}}_{g,n}$: (i) The coefficient of the bouquet class on $\overline{\mathcal{M}}_{g,n}$, which gives the answer  to a conjecture of T.~Kimura and X.~Liu in \cite{kimura2006genus}; (ii) Linear relations among the coefficients of certain rational tails locus of $\overline{\mathcal{M}}_{g,n}$.
  Three applications of topological recursion relations will be discussed: (i) Coefficients of universal equations for Gromov--Witten invariants for any smooth projective variety; (ii) The coefficient of the bouquet class in the double ramification formula of the top Hodge class $\lambda_g$; (iii) A new recursive formula for computing the intersection numbers on the moduli space of stable curves.
\end{abstract}


\maketitle
\tableofcontents
\allowdisplaybreaks
\section{Introduction}
\subsection{Structures of topological recursion relations}
Let $\overline{\mathcal{M}}_{g,n}$ be the moduli space of stable genus-$g$ curves with $n$ marked points.
The tautological rings  $R^*(\overline{\mathcal{M}}_{g,n})$ may be defined simultaneously for all $g,n\geq0$ with $2g-2+n>0$, as the smallest system of subrings of  the Chow ring $A^*(\overline{\mathcal{M}}_{g,n})$ of the moduli spaces of curves closed under pullback via the forgetful morphisms
$\pi\colon \overline{\mathcal{M}}_{g,n+1} \rightarrow \overline{\mathcal{M}}_{g,n}$ and  the gluing morphisms
$\iota_{g_1,g_2}\colon \overline{\mathcal{M}}_{g_1,n_1+1} \times \overline{\mathcal{M}}_{g_2,n_2+1} \rightarrow \overline{\mathcal{M}}_{g_1+g_2, n_1+n_2},$ and
$\iota_{g-1}\colon \overline{\mathcal{M}}_{g,n+2} \rightarrow \overline{\mathcal{M}}_{g+1,n}$. 
This elegant definition is due to Faber and Pandharipande \cite{faber2005relative}, who also proved that $R^*(\overline{\mathcal{M}}_{g,n})$ admits an explicit set of additive generators which are the images of  {\it strata classes} in $\overline{\mathcal{M}
}_{g,n}$.
These strata classes are formed from three ingredients: $\psi$ classes, $\kappa$ classes, and generalized gluing maps corresponding to
stable graphs. The $\psi$ classes are
$\psi_i=c_1(s_i^*(\omega_{\pi})), \; i=1, \ldots, n$
and the $\kappa$ classes are
$\kappa_d = \pi_*(\psi_{n+1}^{d+1}), \; d \geq 0$, where $\pi\colon \overline{\mathcal{M}}_{g,n+1} \rightarrow \overline{\mathcal{M}}_{g,n}$ forgets the last marked point,  $s_i$ is the section of $\pi$ given by the $i$-th marked point and $\omega_{\pi}$ denotes the relative dualizing sheaf.

In the study of the tautological ring, one of the most important problems is to determine the {\it tautological relations}: the linear relations among the strata classes.
In this paper, we focus on linear equations only involving $\psi$ classes and boundary classes, which we refer to as {\it topological recursion relations}.

It is well known that such topological recursion relations produce universal equations for Gromov--Witten invariants.
Such equations play important role in studying properties for Gromov--Witten invariants and its associated integrable systems.
For example they play a crucial role in the study of the low genus Virasoro conjecture (cf.\ \cite{liu2002quantum}) and  can be used to analyze the second Poisson bracket of
Dubrovin-Zhang hierarchy (cf. \cite{liu2011new}, \cite{iglesias2022bi}).

The lower the degree of topological recusion relations, the more powerful are the corresponding universal equations.
From the results in \cite{ionel2002topological}, \cite{faber2005relative}, and \cite{graber2005relative}, we know
that for every $g$, there exist tautological relations that express any degree $g$ polynomial in $\psi$ classes on $\overline{\mathcal{M}}_{g,n}$ in terms of tautological classes supported on the boundary $\overline{\mathcal{M}}_{g,n} \setminus \mathcal{M}_{g,n}$.
Recently, in \cite{clader2023topological}, with E.~Clader and D.~Zakharov, we significantly strengthened this result, and proved the existence of topological recursion relations expressing any degree $g$ polynomial in $\psi$ classes on $\overline{\mathcal{M}}_{g,n}$ in terms of boundary classes.
The proof is constructive, and gives an algorithm to compute its coefficients.
Degree $g$ is the smallest for which such a set of relations can exist.

At this point, all explicit topological recursion relations have genus
$g \leq 4$.
Moreover while the genus $0$ WDVV equations are simple, starting from
Getzler's equation in genus $1$, the known topological recursion
relations are very complicated (cf.\ \cite{getzler1997intersection,getzler1998topological}, \cite{belorousski2000descendent}, \cite{kimura2006genus, kimura2015topological} and \cite{wang2020genus}).

In this paper, we study basic structures of degree-$g$ topological recursion relations in arbitrary genus.

Our first result concerns the coefficient of the \emph{bouquet class} on $\overline{\mathcal{M}}_{g,n}$ in topological recursion relations, that is the class $\left(\xi_{\Gammaloop_{0,g,[n]}}\right)_*(1)$, where $\Gammaloop_{0,g,[n]}$ is the stable graph with only one genus-0 vertex equipped with $g$ loops and all the markings.
\begin{theorem}\label{thm:coeff-bouquet}
For any non-negative integers $\{k_i\}_{i=1}^{n}$ satisfying  $\sum_{i=1}^{n}k_i=g\geq1$, there exists a topological recursion relation on $\overline{\mathcal{M}}_{g,n}$ of the form
\begin{align*}
\prod_{i=1}^{n}{\psi_i}^{k_i} =\frac{1}{8^g \prod_{i=1}^{n}(2k_i+1)!!}
\left(\xi_{\Gammaloop_{0,g,[n]}}\right)_*(1)+\cdots.
\end{align*}
The omitted terms in the above formula satisfy the following condition: there are no
$\kappa$ classes and no genus-$h$ vertices with monomial of $\psi$ classes of degree $\geq h+\delta_{h}^0$ for
$0\leq h\leq g$.
\end{theorem}
Theorem~\ref{thm:coeff-bouquet} gives an affirmative answer to a conjecture about the coefficients  of the bouquet class in topological recursion relations proposed by T.~Kimura and X.~Liu in \cite[p.~656]{kimura2006genus}. In fact, the conjecture about the coefficients  of the bouquet class for $n=1$ was proposed  in \cite{kimura2006genus} and later  generalized to arbitrary $n$ in \cite{liu-private}.

It is also very interesting to compute the coefficients of topological recursion relations corresponding to strata of curves with rational tails.
In general, this is complicated since there are many such coefficients.
However, we do have the following linear relations among these coefficients:
\begin{theorem}\label{thm:trr-g-2}
For any non-negative integers $\{k_i\}_{i=1}^{n}$ satisfying $\sum_{i=1}^{n}k_i=g\geq1$, there exists a topological recursion relation on $\overline{\mathcal{M}}_{g,n}$ of the form\footnote{We use the standard convention $(-1)!!=1$.}
\begin{align*}
  \prod_{i=1}^{n}\psi_i^{k_i}=a_0\xi_{*}(\psi_{\bullet}^{g-1})+\sum_{i<j}a_{ij}\xi_{ij*}(\psi^{g-2}_{\bullet})+\cdots,
\end{align*}
where $a_0$ and the $a_{ij}$ are rational coefficients, and where
\begin{align*}
a_0+\sum_{1\leq i<j\leq n}a_{ij}
=\begin{cases}
1 &g=1, n \ge 2,\\
\frac{(2g-1)!!}{(2k_1-1)!!(2k_2-1)!!}&g\geq2, n=2,\\
0&g\geq2, n\geq3,
\end{cases},
\end{align*}
where $\xi:\overline{\mathcal{M}}_{g,\{\bullet\}}\times\overline{\mathcal{M}}_{0,\bar{\bullet}\cup\{1,\dots,n\}}\rightarrow \overline{\mathcal{M}}_{g,n}$  and $\xi_{ij}:\overline{\mathcal{M}}_{g,\{\bullet\}}\times\overline{\mathcal{M}}_{0,\bar{\bullet}\cup\{1,\dots,\hat{i},\dots,\hat{j},\dots,n\}\cup\circ}\times\overline{\mathcal{M}}_{0,\{\bar{\circ},i,j\}}\rightarrow \overline{\mathcal{M}}_{g,n}$ are the gluing maps.\footnote{Here we set $a_0$ or $a_{ij}$ to be zero if $\xi$ or $\xi_{ij}$ is not defined.}
The omitted terms consist of other boundary classes with no $\kappa$ classes and no genus-0 vertex with $\psi$ classes.
In particular, in the case $n = 2$, the coefficient corresponding to the boundary stratum of curves with rational tails on $\overline{\mathcal{M}}_{g,2}$ is given by
\begin{align*}
a_0=\frac{(2g-1)!!}{(2k_1-1)!!(2k_2-1)!!},\quad g\geq1.
\end{align*}
\end{theorem}

\subsection{Uniqueness of the coefficients in topological recursion relations}

Theorems~\ref{thm:coeff-bouquet} and \ref{thm:trr-g-2} prove the
existence of topological recursion relations of specific shape.
It is an interesting question how unique such relations are.
This would give much further constraints on the structures of
topological recursion relations.

While in general, we expect there to be many topological recursion
relations of the shape of Theorems~\ref{thm:coeff-bouquet} and
\ref{thm:trr-g-2}, we have the following expectations and results
about uniqueness.
\begin{conjecture}
  \label{conj:bouquet-unique}
  Consider a degree $g$ topological recursion relation of the shape
  \begin{equation*}
    0 = c \cdot \left(\xi_{\Gammaloop_{0,g,[n]}}\right)_*(1)+\cdots,
  \end{equation*}
  where $\Gammaloop_{0,g,[n]}$ is the stable graph with only one genus-0
  vertex equipped with $g$ loops and all the markings.
  The omitted terms in the above formula satisfy the following
  condition: all graphs have at least two vertices, and there are no
  $\kappa$ classes and no genus-$h$ vertices with monomial of $\psi$
  classes of degree $\geq h+\delta_{h}^0$ for $0\leq h\leq g$.

  Then, the coefficient $c$ must vanish.
\end{conjecture}
In Theorem~\ref{thm:coeff-bouquet-general}, we give some evidence to this conjecture.

The coefficients corresponding to strata of curves with rational tails have a very interesting close relationship to Faber's intersection number conjecture (cf. \cite[Conjecture~1(c)]{faber1999conjectural}), which was proved in \cite{liu2009proof}, \cite{buryak2011new}, \cite{garcia2022curious}. 
For $n\geq3$, the coefficients corresponding to strata of curves with rational tails are in general not unique.
This can be seen via Belorousski-Pandharipande relation in $R^2(\overline{M}_{2,3})$ (cf. \cite{belorousski2000descendent}).
For $n\geq3$, motivated by Faber's intersection number conjecture, we expect, in the expression of topological recursion relation $\prod_{i=1}^{n}{\psi_{i}}^{k_i}=\text{boundary classes}$, there should be a common factor $\frac{1}{\prod_{i=1}^{n}(2k_i-1)!!}$ in the coefficients of all rational tails loci.
As we at the moment lack explicit formulae of these coefficients, we postpone the study of this interesting problem.

\subsection{Universal equations for Gromov--Witten invariants}Let $X$ be a smooth projective variety and $\{\phi_\alpha: \alpha=1,\dots,N\}$ be a basis of its cohomology ring $H^*(X;\mathbb{C})$ with $\phi_1=\mathbf{1}$ the identity.
Recall that the \emph{big phase space} for Gromov--Witten invariants of $X$ is defined to be $\prod_{n=0}^{\infty}H^*(X;\mathbb{C})$ with standard
basis $\{\tau_{n}(\phi_\alpha): \alpha=1,\dots,N, n\geq0\}$. Let $\{t_n^\alpha\}$ be 
 the coordinates on the big phase space
with respect to the standard basis.
Let $F_g$ be the genus-$g$ generating function, which is a formal power
series of $\mathbf{t}=(t_n^\alpha)$  with coefficients being the  genus-$g$ Gromov--Witten invariants (see Section~\ref{sec:app-gw} for details).
Denote $\langle\langle\tau_{n_1}(\phi_{\alpha_1}),\dots,\tau_{n_k}(\phi_{\alpha_k})\rangle\rangle_g$ for the derivatives
of $F_g$ with respect to the variables $t_{n_1}^{\alpha_1},\dots,t_{n_k}^{\alpha_k}$.  More generally, the $k$-th covariant derivative of the generating functions $F_g$  with respect to the trivial connection on the big phase space
is denoted by tensors $\langle\langle W_1\ldots W_k\rangle\rangle_{g}$, where $W_1, \dots, W_k$ are vector fields.
For convenience, we identify $\tau_n(\phi_\alpha)$ with
the coordinate vector field $\frac{\partial}{\partial t_n^\alpha}$
on $\prod_{n=0}^{\infty}H^*(X;\mathbb{C})$ for $n\geq0$. If $n<0$, $\tau_n(\phi_\alpha)$ is understood to be the $0$
vector field. We also abbreviate $\tau_0(\phi_\alpha)$ by $\phi_\alpha$. Define $\eta=(\eta_{\alpha\beta})$ to be the matrix of the intersection pairing on $H^{*}(X;\mathbb{C})$  in the basis $\{\phi_{1},\dots,\phi_{N}\}$. We will use
$\eta=(\eta_{\alpha\beta})$ and $\eta^{-1}=(\eta^{\alpha\beta})$ to lower and  raise indices, for example $\phi^{\alpha}:=\eta^{\alpha\beta}\phi_{\beta}$ for any $\alpha$.
Here we use the summation convention that repeated indices should be summed over their entire ranges.

%
%
%
For any vector fields $W_{1}$ and $W_{2}$ on the big phase space, the {\it quantum product} of $W_{1}$ and $W_{2}$ is defined by
$$W_{1}\bullet W_{2}:=\sum_{\alpha}\langle\langle{W_{1}W_{2}
\phi^{\alpha}}\rangle\rangle_0\phi_{\alpha}.$$
Define
the operator $T$ on the space of vector fields by
\[T(W)=\tau_{+}(W)-\sum_{\alpha}\langle\langle W\phi^\alpha\rangle\rangle_0\phi_\alpha\]
where  $\tau_{+}$ is the operator which increases  the level of descendants by one, i.e.\ $\tau_+(\tau_k(\phi)) = \tau_{k + 1}(\phi)$.
The operator $T$ is very useful for
the translating topological recursion relations  into universal equations (cf.\ \cite{liu2006gromov}). 
It is well known that such universal equations play an important role in computing higher genus Gromov--Witten invariants and studying the famous Virasoro conjecture (cf.\ \cite{liu2002quantum}).
It is also conjectured that all such equations can determine all higher genus Gromov--Witten invariants  in terms of genus-0 Gromov--Witten invariants in the semisimple case (cf.\ \cite{liu2006gromov}), which differs from the approach of Givental (cf.\ \cite{givental2001semisimple}) or Dubrovin-Zhang (cf.\ \cite{dubrovin1998bihamiltonian}).   
Thus finding as explicit as possible formulas for such universal  equations is very important, even though they are very complicated in general.

As a byproduct of Theorem~\ref{thm:coeff-bouquet}, we have
\begin{corollary}[Corollary~\ref{cor:bouquet-correlation-fun}]
For any smooth projective variety and  non-negative integers $\{k_i\}_{i=1}^{n}$ satisfying $\sum_{i=1}^{n}k_i=g\geq1$, there exists a topological recursion relation between its Gromov--Witten invariants of the form
\begin{align*}
\langle\langle T^{k_1}(W_1)\ldots T^{k_n}(W_n)\rangle\rangle_g =\frac{1}{8^g \prod_{i=1}^{n}(2k_i+1)!!}\sum_{\alpha_1,\dots,\alpha_{g}=1}^{N}\langle\langle W_1\ldots W_n\phi_{\alpha_1}\phi^{\alpha_1}\ldots \phi_{\alpha_g}\phi^{\alpha_g}\rangle\rangle_0+\cdots
\end{align*}
 The omitted terms in the above formula are polynomials of tensors $\{\langle\langle\ldots\rangle\rangle_h: 0\leq h\leq g\}$ which satisfy the following condition: there are no
 genus-$h$ tensors $\langle\langle\ldots\rangle\rangle_h$  with insertions  of operator $T$  of degree $\geq h+\delta_{h}^0$ for
$0\leq h\leq g$.
\end{corollary}

For a function $f$  on the big phase space, we say $f\overset{g}{\approx}0$ if $f$ can be
expressed as a polynomial function of $F_0,F_1,\dots,F_{g-1}$ and their derivatives. Then
Theorem~\ref{thm:trr-g-2} implies:
\begin{corollary}[Corollary~\ref{cor:two-pt-correlation-fun}]
For any smooth projective variety and any non-negative integers $k_1, k_2$ such that $k_1+k_2=g\geq1$, 
 there exists a topological recursion relation between its Gromov--Witten invariants of the form
\begin{align*}
\langle\langle T^{k_1}(W_1)T^{k_2}(W_2)\rangle\rangle_g  -\frac{(2g-1)!!}{(2k_1-1)!! (2k_2-1)!!}\langle\langle T^{g-1}(W_1\bullet W_2)\rangle\rangle_g\overset{g}{\approx}0.  
\end{align*}
\end{corollary}

\subsection{Coefficient of the bouquet class  in the DR formula of the Hodge class $\lambda_g$}

In \cite{janda2017double}, a nice formula
for the  top Chern class of Hodge bundle $\lambda_g$ was found, which is supported on the divisor of curves with a nonseparating node:
\begin{align*}
\lambda_g=\frac{(-1)^g}{2^g} \mathcal{D}_{g,n}^{g}(0,\dots,0) \in R^{g}(\overline{\mathcal{M}}_{g,n})
\end{align*}
where $\mathcal{D}_{g,n}(a_1,\dots,a_n)$ is the Pixton's formula for double ramification cycles defined in Section~\ref{subsec: pixton-formula}. 
As an application of Theorem~\ref{thm:coeff-bouquet}, we may compute the coefficient of the bouquet class in this formula for $\lambda_g$.

\begin{proposition}
[Proposition~\ref{prop:lambda_g}]
    On $\overline{\mathcal{M}}_{g,n}$, the top Chern class of Hodge bundle may be expressed as
\begin{align*}
\lambda_g=\left[\frac{(-1)^g}{2^g}\frac{t^2e^t}{(e^t-1)^2}\right]_{t^{2g}}
\left(\xi_{\Gammaloop_{0,g,[n]}}\right)_*(1)+\cdots    
\end{align*}
where $\Gammaloop_{0,g,[n]}$ is the stable graph with only one genus-0 vertex equipped with $g$ loops and all the markings.
The omitted terms in the above formula satisfy the following condition: there are no
$\kappa$ classes and no genus-$h$ vertices with monomial of $\psi$ classes of degree $\geq h+\delta_{h}^0$ for
$0\leq h\leq g$. $[\quad]_{t^{2g}}$ means taking the coefficient of $t^{2g}$ in the expressions.
\end{proposition}
\subsection{ Intersection numbers on moduli space of curves}
 
The Gromov--Witten theory of a point, or equivalently, the intersection numbers of the moduli space of stable
curves has been known since Kontsevich's famous proof of Witten's conjecture (\cite{witten1990two}).
There are two parts in Kontsevich's original proof.
First, he reduced the geometric problem to a combinatorial problem, using a topological cell decomposition of the moduli space of
curves to derive formulae for integrals of tautological $\psi$ classes.
Then he derived a matrix integral formula for these expressions, and used this to prove Witten's conjecture.

The degree $3g-3+n$ topological recursion relations on $\overline{\mathcal{M}}_{g,n}$, applied to a point, give a new way to do the first
part of this procedure; that is we find a  new combinatorial recursive  formula for integrals
of tautological $\psi$ classes on the moduli space of curves. 
\begin{theorem}
[Theorem~\ref{thm:algo-inter}]
 For any $\sum_{i=1}^{n}k_i=3g-3+n$,
\begin{align*}
\int_{\overline{\mathcal{M}}_{g,n}}  \prod_{i=1}^{n}{\psi_i}^{k_i}   =\sum_{\substack{\Gamma\in G_{g,n}\\ |E(\Gamma)|\geq1}}\sum_{\{\alpha_v:v\in V(\Gamma)\}} c_{(\Gamma,\{\alpha_v\}_{v\in V(\Gamma)})}\cdot \prod_{v\in V(\Gamma)}\int_{\overline{\mathcal{M}}_{g(v),n(v)}}\alpha_{v}  
\end{align*}
where $\alpha_v$ ranges over all possible  monomials of tautological $\psi$ classes on $\overline{\mathcal{M}}_{g(v),n(v)}$. Each coefficient $c_{(\Gamma,\{\alpha_v\}_{v\in V(\Gamma)})}$ is a rational number coming from the weight on stable graphs in Pixton's double ramifcation formula and factors from string equation and dilaton equation.        
\end{theorem}
In principle, by induction on  $g$ and $n$, this theorem provides us a combinatorial recursion formula for computing the intersection numbers $\int_{\overline{\mathcal{M}}_{g,n}} \prod_{i=1}^{n}{\psi_i}^{k_i}$ for any $2g-2+n>0$ from initial condition $ \int_{\overline{\mathcal{M}}_{0,3}}1=1$ and $\int_{\overline{\mathcal{M}}_{1,1}}\psi_1=\frac{1}{24}$.
The techniques we use are purely  combinatorial from Pixton's double ramification relation, and thus we obtain  a very different approach from Kontsevich's topological model.

It would be very interesting to try to recover Witten's conjecture from the algorithm in Theorem~\ref{thm:algo-inter}.

\medskip 

This paper is organised as follows:
In Section~\ref{sec:preliminaries}, we review the tautological ring and the algorithm from \cite{clader2023topological} for computing topological recursion relations.
In Section~\ref{sec:coeff-bou}, we study the coefficients of the bouquet class in topological  recursion relations and prove Theorem~\ref{thm:coeff-bouquet}.
In Section~\ref{sec:boun-g-2}, we study the locus of curves with rational tails in topological  recursion relations and prove Theorem~\ref{thm:trr-g-2}.
In Section~\ref{sec:app-gw},  we give some direct applications in Gromov--Witten theory.
In Section~\ref{sec:appl-lambda-g}, we compute the coefficient of the bouquet class in the DR formula for the top Hodge class $\lambda_g$.
In Section~\ref{sec:inter-number-mgn}, we propose a recursive algorithm for  computing intersection numbers on the moduli space of stable curves.

\subsection*{Acknowledgements}

The authors would like to thank E.~Clader and D.~Zakharov for the collaboration \cite{clader2023topological} and for discussions related to this work.
They also thank the anonymous referee for carefully reading the manuscript, and for making helpful suggestions to improve the clarity of this paper.
They are especially grateful to X.~Liu for sharing his generalized conjecture on the coefficient of the bouquet class.
The first author was partially supported by NSF grants DMS-2054830 and DMS-2239320.
The second author was partially supported by  National
Science Foundation of China (Grant No. 12071255) and Shandong Provincial Natural Science Foundation
(Grant No. ZR2021MA101).

\section{The tautological ring and topological recursion relations}
\label{sec:preliminaries}

\subsection{The strata algebra and the tautological ring of $\overline{\mathcal{M}}_{g,n}$}
\label{sec:strata}

The additive generators of the tautological ring
of $\overline{\mathcal{M}}_{g,n}$ are described by the strata algebra.
In this section, we review basic definitions following \cite[\S 2.1]{clader2023topological}, and refer the reader to \cite{graber2003constructions} and \cite{pandharipande2015relations} for more details.

A {\it stable graph} $\Gamma=(V,H,E,L,g,p,\iota,m)$ of genus $g$ with $n$ legs consists of the following data:
\begin{enumerate}
\item a finite set of vertices $V$  with a genus function $g\colon V\to \mathbb{Z}_{\geq 0}$;

\item a finite set of half-edges $H$  with a vertex assignment $p\colon H\to V$ and an involution $\iota\colon H\to H$;

\item a set of edges $E$, which is the set of two-point orbits of $\iota$;

\item a set of legs $L$, which is the set of fixed points of $\iota$, and which is marked by a bijection $m\colon\{1,\ldots,n\}\to L$.

\item The graph $(V,E)$ is connected.

\item For every vertex $v\in V$, we have
\begin{equation*}
  2g(v)-2+n(v) > 0,
\end{equation*}
where $n(v)=|  p^{-1}(v)|$ is the {\it valence} of the vertex $v$
\item The genus of the graph is $g$, in the sense that
\begin{equation*}\label{eq:genus}
g=h^1(\Gamma)+\sum_{v\in V}g(v),
\end{equation*}
where $h^1(\Gamma)= |E|- |V|+1$.
\end{enumerate}

An automorphism of a stable graph $\Gamma$ consists of permutations of the sets $V$ and $H$ which
leave invariant the structures  $g$, $p$, $\iota$ and $m$ (and hence preserve $L$ and $E$). We denote by $\Aut(\Gamma)$ the automorphism group of $\Gamma$.

To each stable graph $\Gamma$, let
\begin{equation*}
\overline{\mathcal{M}}_{\Gamma}:=\prod_{v\in V}\overline{\mathcal{M}}_{g(v),n(v)}.
\end{equation*}
There is a generalized  gluing map
\begin{equation}\label{eq:xi}
\xi_{\Gamma}\colon\overline{\mathcal{M}}_{\Gamma}\to \overline{\mathcal{M}}_{g,n},
\end{equation}
whose image is the locus in $\overline{\mathcal{M}}_{g,n}$ having generic point corresponding to a curve with stable graph $\Gamma$. To
construct $\xi_{\Gamma}$, a family of stable pointed curves over $\overline{\mathcal{M}}_{\Gamma}$ is required. Such a
family is easily defined by attaching the pull-backs of the universal families
over each of the $\overline{\mathcal{M}}_{g(v),n(v)}$ along the sections corresponding to half-edges.  The degree of $\xi_{\Gamma}$, as a map of Deligne--Mumford stacks, is equal to $|\Aut(\Gamma)|$.

Additive generators of the tautological ring can be described in terms of certain decorations on stable graphs $\Gamma$.  Namely, let
\begin{equation*}
\gamma=(x_i:V\to \mathbb{Z}_{\geq 0},y:H\to \mathbb{Z}_{\geq 0})
\end{equation*}
be a collection of functions such that
\begin{equation*}
d(\gamma_v)=\sum_{i>0}ix_i[v]+\sum_{h\in p^{-1}(v)} y[h]\leq 3g(v)-3+n(v)
\end{equation*}
for all $v\in V$.  Then, for each $v$, define
\begin{equation*}
\gamma_v=\prod_{i>0}\kappa_i^{x_i[v]}\prod_{h\in p^{-1}(v)}
\psi_h^{y[h]}\in A^{d(\gamma_v)}(\overline{\mathcal{M}}_{g(v),n(v)}).
\end{equation*}
Associated to any such choice of decorations $\gamma$, there is a {\it basic class} on $\overline{\mathcal{M}}_{\Gamma}$, also denoted by $\gamma$, defined by
\begin{equation*}
\gamma=\prod_{v\in V}\gamma_v \in A^{d(\gamma)}(\overline{\mathcal{M}}_{\Gamma}).
\end{equation*}
Here, the degree is $d(\gamma):=\sum_{v\in V}d(\gamma_v)$, and we abuse notation slightly by using a product over classes on the vertex moduli spaces to denote a class on $\overline{\mathcal{M}}_{\Gamma}$.

The {\it strata algebra} $\mathcal{S}_{g,n}$ is defined to be  the finite dimensional $\mathbb{Q}$-vector space spanned by isomorphism classes of pairs $[\Gamma,\gamma]$, where $\Gamma$ is a stable graph of genus $g$ with $n$ legs and $\gamma$ is a basic class on $\overline{\mathcal{M}}_{\Gamma}$.  The product is defined by intersection theory (see~\cite{faber1999algorithms}).  More precisely, for $[\Gamma_1,\gamma_1], [\Gamma_2,\gamma_2]\in \mathcal{S}_{g,n}$, the fiber product of $\xi_{\Gamma_1}$ and $\xi_{\Gamma_2}$ over $\overline{\mathcal{M}}_{g,n}$ is a disjoint union of $\xi_{\Gamma}$ over all graphs $\Gamma$ having edge set $E=E_1\cup E_2$, such that $\Gamma_1$ is obtained by contracting all edges outside of $E_1$ and $\Gamma_2$ is obtained by contracting all edges outside of $E_2$. We then define the product by
\begin{align}\label{eqn:stra-prod}
[\Gamma_1,\gamma_1]\cdot [\Gamma_2,\gamma_2]=\sum_{\Gamma}\left[\Gamma,\gamma_1\gamma_2\prod_{(h,h')\in E_1\cap E_2}(-\psi_h-\psi_{h'})\right].
\end{align}

Pushing forward elements of the strata algebra along the gluing maps \eqref{eq:xi} defines a ring homomorphism
\begin{align*}
&q\colon\mathcal{S}_{g,n}\to A^*(\overline{\mathcal{M}}_{g,n})\\
&q([\Gamma,\gamma])=\xi_{\Gamma*}(\gamma),
\end{align*}
and the image of $q$ is precisely the tautological ring $R^*(\overline{\mathcal{M}}_{g,n})$.  Elements of $\mathcal{S}_{g,n}$ in the kernel of $q$ are referred to as {\it tautological relations}. Let $\mathcal{S}^{trr}_{g,n}\subset \mathcal{S}_{g,n}$ denote the subalgebra spanned by classes which do not have any $\kappa$ classes. Then the elements of $\mathcal{S}^{trr}_{g,n}$ in the kernel of $q$ are referred to as {\it topological recursion relations}. 

For each basis element $[\Gamma,\gamma]$, we define its degree as
\begin{equation*}
\deg[\Gamma,\gamma]=|E|+d(\gamma).
\end{equation*}
Since product \eqref{eqn:stra-prod} preserves the degree, then  $\mathcal{S}_{g,n}$ is a graded ring:
\begin{equation*}
\mathcal{S}_{g,n}=\bigoplus_{d=0}^{3g-3+n}\mathcal{S}_{g,n}^d.
\end{equation*}




\subsection{Pixton's DR formula}\label{subsec: pixton-formula}

Fix $g$ and $n$, and fix a collection of integers $A = (a_1, \ldots, a_n)$ such that $\sum_j a_j =0$.  In this subsection, we recall the definition of Pixton's class, which is an inhomogeneous element of $\mathcal{S}_{g,n}$ depending on $A$.

For our purpose,  we  first define auxiliary classes $\widetilde{\mathcal{D}}_{g,n}^r$, for an additional positive integer parameter $r$.  For a stable graph $\Gamma=(V,H,g,p,\iota)$ of genus $g$ with $n$ legs, a \emph{weighting modulo $r$} on $\Gamma$ is defined to be a map
\begin{equation*}
  w\colon H \to \{0, \dotsc, r - 1\}
\end{equation*}
satisfying three properties:
\begin{enumerate}
\item For any $i \in \{1, \dotsc, n\}$ corresponding to a leg $\ell_i$ of $\Gamma$, we have $w(\ell_i) \equiv a_i \pmod{r}$.
\item For any edge $e\in E$ corresponding to two half-edges $h, h'\in H$, we
  have $w(h) + w(h') \equiv 0 \pmod{r}$.
\item For any vertex $v \in V$, we have
  $\sum_{h\in p^{-1}(v)} w(h) \equiv 0 \pmod{r}$.
\end{enumerate}
Define $\widetilde{\mathcal{D}}_{g,n}^r$ to be the class
\begin{equation*}
  \sum_{\Gamma\in G_{g,n}} \frac{1}{|\Aut(\Gamma)|} \frac 1{r^{h^1(\Gamma)}}  \sum_{\substack{w \text{ weighting }\\ \text{mod } r \text{ on }\Gamma}} \left[\Gamma, \prod_{i=1}^n e^{\frac 12 a_i^2\psi_i} \prod_{(h,h')\in E} \frac{1 - e^{-\frac 12 w(h)w(h')(\psi_h + \psi_{h'})}}{\psi_h + \psi_{h'}}\right] \in \mathcal{S}_{g,n}.
\end{equation*}

Pixton proved the class $\widetilde{\mathcal{D}}_{g,n}^r$ is a polynomial in
$r$ for $r$ sufficiently large (see \cite{Pixton2024}).  Then Pixton's class is defined to be  the constant term of this polynomial. Furthermore,  by the fact that
\[a_1 = -(a_2 + \cdots +a_n),\]
Pixton's class can be expressed in terms of the variables $\{a_i\}_{i=2}^{n}$, then we denote it by
\[\mathcal{D}_{g,n}(a_2, \ldots, a_n),\]
and  its degree $d$ component  is denoted by $\mathcal{D}^d_{g,n}(a_2, \ldots, a_n)$.
The following statement was conjectured by Pixton and proved in \cite{clader2018pixton}:
\begin{theorem}[\cite{clader2018pixton}]
\label{thm:DR}
For each $d > g$, $\mathcal{D}^d_{g,n}(a_2, \ldots, a_n)$ is a tautological relation.
\end{theorem}


\subsection{Polynomiality properties}
\label{subsec:poly}

In fact, Pixton proved the following result about the polynomiality of $\mathcal{D}_{g,n}(a_2, \ldots, a_n)$, which Spelier provided an alternative proof:
\begin{theorem}[Pixton \cite{Pixton2024}, Spelier \cite{spelier2024polynomiality}]
\label{thm:poly}
The class $\mathcal{D}_{g,n}(a_2, \ldots, a_n)$ depends polynomially on $a_2, \ldots, a_n$.
\end{theorem}
Combining Theorems~\ref{thm:DR} and \ref{thm:poly}, the coefficient of any monomial $a_2^{b_2} \cdots a_n^{b_n}$ in the class $\mathcal{D}_{g,n}^d(a_2, \ldots, a_n)$ yields a tautological relation in $R^d(\overline{\mathcal{M}}_{g,n})$, for each $d > g$:
\begin{equation*}
q\left( \bigg[\mathcal{D}_{g,n}^d(a_2, \ldots, a_n)\bigg]_{a_2^{b_2} \cdots a_n^{b_n}} \right) = 0.
\end{equation*}






\subsection{An algorithm for deriving topological recursion relations on $\overline{\mathcal{M}}_{g,n}$ }\label{subsec:algorithm-trr}
First, we recall the algorithm for computing general degree-$g$ topological recursion relations on $\overline{\mathcal{M}}_{g,n}$ from \cite{clader2023topological}.

Fix a genus $g$ and a number of marked points $n$.  Let $M$ be a monomial of degree $D\leq 2g+1$ in the variables $a_2, \ldots, a_n$, and let
\begin{align*}
N:= n+2g+2-D.
\end{align*}
Define
\[\Omega_{g,M}^{\text{pre}}=\left[\mathcal{D}^{g+1}_{g,N}(a_2, \ldots, a_N)\right]_{M\cdot a_{n+1}\dots a_{N}}
 \in \mathcal{S}^{g+1}_{g,N}\]
to be the coefficient of the monomial $M \cdot a_{n+1} \cdots a_N$ in Pixton's class $\mathcal{D}^{g+1}_{g,N}(a_2, \ldots, a_N)$.  Note that we use here the polynomiality discussed in Section~\ref{subsec:poly}.
Let
\[\Pi\colon \overline{\mathcal{M}}_{g,N}=\overline{\mathcal{M}}_{g,n+1+(2g+1-D)} \rightarrow \overline{\mathcal{M}}_{g,n+1}\]
be the forgetful map, and define
\[\Omega_{g,M}:= \Pi_*(\Omega_{g,M}^{\text{pre}} \cdot \psi_{n+2} \cdots \psi_{N}) \in \mathcal{S}_{g,n+1}^{g+1},\]
where we use $\Pi_*$ to denote the induced map on strata algebras.
Via the dilaton equation and string equation, pushing-forward $\Omega_{g,M}$
along the forgetful map
$\pi\colon \overline{\mathcal{M}}_{g,n+1} \rightarrow
\overline{\mathcal{M}}_{g,n}$, we obtain a degree-$g$ topological
recursion relation.

\medskip

Topological recursion relations which express any specific degree $g$
monomial $\psi_1^k \prod_{j=2}^n\psi_j^{k_j}$ of $\psi$ classes in terms of
boundary classes, are obtained recursively in terms of linear combinations of
the above topological recursion relations.
To describe them explicitly, we define an order on the set of monomials
$\{\psi_1^k \prod_{j=2}^n \psi_i^{l_j}: 
k+\sum_{j=2}^{n}l_j=g\}$: for any monomials
\begin{equation}
\label{eq:psimon}
M=\psi_1^k \prod_{j=2}^n \psi_i^{l_j},\quad M'=\psi_1^{k'} \prod_{j=2}^{n} \psi_i^{l_j'}
\end{equation}
of degree $g$ in the $\psi$ classes on $\overline{\mathcal{M}}_{g,n}$,
we say that the $\psi$-monomial $M'$ is \emph{lower} than
$M$ if  $l_j' \le l_j$ for $j \ge 2$ with at
least one of this inequalities being strict.

We consider the linear combination on $\overline{\mathcal{M}}_{g,n}$
\begin{equation}
  \label{eq:cancellation}
  \sum_{d_2, \dotsc, d_n = 0}^1
  \frac{\prod_{j = 2}^n (-2l_j)^{1 - d_j}}{(2g + 1 - \sum_{j=2}^{n} (2l_j - d_j))!} \cdot \pi_* \Omega_{g, a_2^{2l_2 - d_2} \cdot \dotsb \cdot a_n^{2l_n - d_n}}=0
\end{equation}
defined when $l_j \ge 1$ for all $j$.
By \cite[Proposition~13]{clader2023topological}, up to a constant multiple, the
combination \eqref{eq:cancellation} gives a topological recursion
relation of the form
\begin{equation}
\label{eqn:want}
\psi_1^{g-\sum_{j=2}^{n} l_j} \prod_{j=2}^n \psi_j^{l_j}
+\sum_{(k_2,\dots,k_n)<(l_2,\dots,l_n)}c_{k_2,k_3,\dots,k_n}\psi_1^{g-\sum_{j=2}^{n}k_j} \prod_{j=2}^n \psi_j^{k_j}=\text{boundary terms}
\end{equation}
where $(k_2,\dots,k_n)<(l_2,\dots,l_n)$ refers to the monomials of $\psi$ classes $\psi_1^{g-\sum_{j=2}^{n}k_j}\prod_{j=2}^{n}\psi_j^{k_j}$ strictly lower than  $\psi_1^{g-\sum_{j=2}^{n}l_j}\prod_{j=2}^{n}\psi_j^{l_j}$, and all $c_{k_2,k_3,\dots,k_n}$ are certain rational numbers determined from Equation~\eqref{eq:cancellation}.
Also note that being topological recursion relations, the boundary terms in \eqref{eqn:want} do not involve any $\kappa$-classes.
Using \eqref{eqn:want} repeatedly yields:
\begin{corollary}\label{cor:deg-d-TRR-uniform}
  For any $\psi$-monomial $\psi_1^{k_1} \cdots \psi_n^{k_n}$ of degree $d \ge g$, there exists a linear combination of the topological recursion relations $\pi_* \Omega_{g, M}$ for various monomials $M$, as well as gluing map push-forwards of such relations of lower genera, which is a topological recursion relation of the form
  \begin{equation*}
    \prod_{j=1}^n \psi_j^{k_j} = \text{boundary terms}.
  \end{equation*}
  The boundary terms in the above formula satisfy the following condition: there are no $\kappa$ classes and no genus-$h$ vertices with monomial of $\psi$ classes of degree $\geq h+\delta_{h}^0$ for $0\leq h\leq g$.
\end{corollary}

\subsection{Generating series formula for $\Omega_{g, M}$}

In this section, we prove a simplified formula for the coefficients of
$\Omega_{g, M}$.

First, fix integers $(a_2, \dotsc, a_n)$, an integer $r > 0$, a
stable graph $\Gamma=(V,H,g,p,\iota)$ of genus $g$ with $n + 1$ legs, and  an additional assignment $x_v \in \mathbb Z$ for
every $v \in V$.
Given this, we set
\begin{equation*}
  a_1 = -(a_2 + \dotsb + a_n + \sum_{v \in V} x_v).
\end{equation*}
Then, we define a \emph{weighting modulo $r$} on $(\Gamma, x)$ to be a
map
\begin{equation*}
  w\colon H \to \{0, \dotsc, r - 1\}
\end{equation*}
satisfying three properties:
\begin{enumerate}
\item For any $i \in \{1, \dotsc, n\}$ corresponding to a leg $\ell_i$
  of $\Gamma$, we have $w(\ell_i) \equiv a_i \pmod{r}$.
  We set $w(\ell_{n + 1}) = 0$.
\item For any edge $e\in E$ corresponding to two half-edges $h, h'\in H$, we
  have $w(h) + w(h') \equiv 0 \pmod{r}$.
\item For any vertex $v \in V$, we have
  $\sum_{h\in p^{-1}(v)} w(h) \equiv -x_v \pmod{r}$.
\end{enumerate}
\begin{remark}
  It is remarkable that this notion of weighting is very similar to the weightings modulo $r$ in \cite[Definition~5]{bae2023pixton}.    
\end{remark}

Define $\widetilde{\mathcal{D}}_{\Gamma}^r$ to be the class
\begin{equation*}
  \frac 1{r^{h^1(\Gamma)}} \sum_{\substack{w \text{ weighting }\\ \text{mod } r \text{ on }\Gamma}} [\Gamma, \gamma_w] \in \mathcal{S}_{g,n + 1},
\end{equation*}
where
\begin{equation*}
  \gamma_w = \prod_{i=1}^n e^{\frac 12 a_i^2\psi_i} \prod_{(h,h')\in E} \frac{1 - e^{-\frac 12 w(h)w(h')(\psi_h + \psi_{h'})}}{\psi_h + \psi_{h'}}.
\end{equation*}
This class is a polynomial for $r \gg 0$, and we will use
$\mathcal{D}_{\Gamma}$ to denote the constant part of the resulting
polynomial.
The class $\mathcal{D}_{\Gamma}$ is a polynomial in the variables
$a_2, \dotsc, a_n$ and $x_v$ for $v \in V$, and so we will
denote it by
\begin{equation*}
  \mathcal{D}_{\Gamma}(a_2, \dotsc, a_n, \{x_v\}).
\end{equation*}

\begin{proposition}\label{prop:omega_g,m-formula}
  For any monomial $M$ of degree $D \le 2g + 1$, we have
  \begin{align}\label{eqn:omega_g,m-formula}
    \Omega_{g, M}
    = \sum_{\Gamma\in G_{g,n+1}} \sum_{\substack{m\colon V \to \mathbb Z_{\ge 0} \\ \sum_v m(v) = 2g + 1 - D}}
    \frac{(m(v^*) + 1)(2g + 1 - D)!}{|\Aut(\Gamma)|}
    \prod_{v \in V} \frac{(2g(v) - 3 + n(v) + m(v))!}{(2g(v) - 3 + n(v))!} \cdot C_{\Gamma, M, m},
  \end{align}
  where $v^*$ is the vertex containing the $(n + 1)$st leg, and where
  $C_{\Gamma, M, m}$ is the coefficient of
  $M \cdot x_{v^*} \prod_v x_v^{m(v)}$ of the degree-$(g + 1)$ part of
  the class $\mathcal{D}_{\Gamma}(a_2, \dotsc, a_n, \{x_v\})$.
\end{proposition}
\begin{proof}
  In the computation of $\Omega_{g, M}$, we first split $\mathcal D_{g, N}^{g+1}(a_1, \dotsc, a_N)$ into the contributions of each stable graph $\widetilde\Gamma$
  \begin{equation*}
    \mathcal D_{g, N}^{g+1}(a_1, \dotsc, a_N)
    = \sum_{\widetilde\Gamma \in G_{g, N}} \frac 1{|\Aut(\widetilde\Gamma)|}\mathcal D_{\widetilde\Gamma}(a_1, \dotsc, a_N).
  \end{equation*}
  Note that under the map $\Pi$ forgetting the last $2g + 1 - D$ markings $\{n+2,\dots,N\}$, each graph $\widetilde\Gamma \in G_{g, N}$ maps to a graph $\Gamma \in G_{g, n + 1}$.
  In the other direction, because of the multiplication by $\psi_{n + 2} \cdot \dots \cdot \psi_N$ in the definition of $\Omega_{g, M}$, given a graph $\Gamma \in G_{g, n + 1}$ all graphs $\widetilde\Gamma \in G_{g, N}$ with nonzero contributions that map to it are obtained by adding additional legs to the vertices of $\Gamma$.

  Using this, we may write $\Omega_{g, M}$ as a sum over graphs $\Gamma \in G_{g, n + 1}$.
  Given $\Gamma \in G_{g,n+1}$ and a choice of set partition $\{n+2, \ldots, N\} = \sqcup_{v\in V} I_v$, we denote by $\widetilde{\Gamma}^{\sqcup_{v\in V}I_v} \in G_{g, N}$ the graph obtained by adding $I_v$ additional legs on each vertex $v$ of $\Gamma$.
  Further, define
  \begin{equation}
    \label{eq:xv-def}
    x_v =
    \begin{cases}
      \sum_{i \in I_v} a_i & \text{if } v \neq v^* \\
      a_{n + 1} + \sum_{i \in I_v} a_i & \text{if } v = v^*.
    \end{cases}
  \end{equation}
  Note that the multiplication by $\psi_{n + 2} \cdot \dots \cdot \psi_N$
  in the definition of $\Omega_{g, M}$ ensures that
  we can apply the dilaton equation on each vertex repeatedly to compute the
  $\Pi$-push-forward of $\Omega_{g,M}^{\text{pre}}\cdot\prod_{i=n+2}^{N}\psi_{i}$.
  Putting this together, we arrive at
  \begin{equation}
    \label{omega_g,m-formula-part1}
    \Omega_{g, M}
    = \sum_{\Gamma\in G_{g,n+1}} \sum_{\{n+2, \ldots, N\} = \sqcup_{v\in V} I_v}
    \frac 1{|\Aut(\widetilde\Gamma^{\sqcup_{v\in V}I_v})|}
    \prod_{v \in V} \frac{(2g(v) - 3 + n(v) + |I_v|)!}{(2g(v) - 3 + n(v))!} \cdot C_{\Gamma, M, I},
  \end{equation}
  where $C_{\Gamma, M, I}$ is the coefficient of $M \cdot a_{n + 1} \cdots a_N$ of the degree $(g + 1)$-part of the class $\mathcal{D}_{\Gamma}(a_2, \dotsc, a_n, \{x_v\})$, and the $x_v$ are as in \eqref{eq:xv-def}.

  Because of the identity
  \begin{equation*}
    [p(x_1 + \dotsb + x_m)]_{x_1 \cdot \dotsb \cdot x_m}
    = m! [p(x)]_{x^m}
  \end{equation*}
  for any univariate polynomial $p(x)$, we see that
  \begin{equation}
    \label{omega_g,m-formula-part2}
    C_{\Gamma, M, I} = (m(v^*) + 1)! \prod_{v \neq v^*} m(v)! \cdot C_{\Gamma, M, m},
  \end{equation}
  writing $m(v) = |I_v|$ for each vertex $v$ on $\Gamma$.

  We may then conclude the proof of the proposition by combining
  \eqref{omega_g,m-formula-part1} and \eqref{omega_g,m-formula-part2},
  and making two further observations.
  First, the orbit-stabilizer theorem implies that given a graph
  $\widetilde\Gamma \in G_{g, N}$ (with nonzero contribution) and
  corresponding graph $\Gamma \in G_{g, n + 1}$, there are exactly
  \begin{equation*}
    \frac{|\Aut(\Gamma)|}{|\Aut(\widetilde\Gamma)|},
  \end{equation*}
  many choices of set partitions
  $\{n+2, \ldots, N\} = \sqcup_{v\in V} I_v$ such that the graph
  $\widetilde\Gamma^{\sqcup_{v\in V}I_v}$ formed from $\Gamma$ is
  isomorphic to $\widetilde\Gamma$.
  Second, given a function $m\colon V \to \mathbb{Z}_{\ge 0}$, there
  are
  \begin{equation*}
    \frac{(2g + 1 - D)!}{\prod_v m(v)!}
  \end{equation*}
  many ways to choose $I_v$ such that $m(v) = |I_v|$.
\end{proof}
\begin{remark}\label{rem:contr-asimple-loop}
  The contribution of the dual graph on $\overline{\mathcal{M}}_{g,n+1}$ with a simple loop in equation~\ref{eqn:omega_g,m-formula} is trivial.
  In fact, the corresponding coefficient  of $M\cdot x_{v^*}\prod_{v}x_v^{m(v)}$ of the class $\mathcal{D}_{\Gamma}(a_2,\dots ,a_n,\{x_v\})$ has degree
strictly bigger than $g+1$.
\end{remark}
\subsection{The push-forward formula of $\pi_*(\Omega_{g,M})$}

We now compute the $\pi$-push-forward of the formula $\Omega_{g,M}$ in Proposition~\ref{prop:omega_g,m-formula} from the previous
section.

First, note that given a dual graph $\Gamma$ on
$\overline{\mathcal{M}}_{g,n}$, there are three types of dual graphs
$\Gamma'$ of $\overline{\mathcal{M}}_{g,n + 1}$ mapping to it:
\begin{enumerate}
\item $\Gamma'$ is obtained from $\Gamma$ by adding an extra $(n+1)$-th  leg to
  one of the vertices of $\Gamma$.
\item $\Gamma'$ is obtained from $\Gamma$ by replacing the $i$-th leg
  by an edge connected a new genus zero vertex with the $i$-th and
  $(n+1)$-st leg.
\item $\Gamma'$ is obtained from $\Gamma$ by dividing an edge into two
  edges, and putting the $i$-th leg on the resulting new genus zero
  vertex.
\end{enumerate}
Fix integers $(a_2, \dotsc, a_n)$, an integer $r > 0$,  a
stable graph $\Gamma=(V,H,g,p,\iota)$ of genus $g$ with $n$ legs and 
an  assignment $\{x_v \in \mathbb Z: v \in V\}$.
Given this, we set
\begin{equation*}
  a_1 = -(a_2 + \dotsb + a_n + \sum_{v \in V} x_v).
\end{equation*}
Then, we define a \emph{weighting modulo $r$} on $(\Gamma, x)$ to be a
map
\begin{equation*}
  w\colon H \to \{0, \dotsc, r - 1\}
\end{equation*}
satisfying three properties:
\begin{enumerate}
\item For any $i \in \{1, \dotsc, n\}$ corresponding to a leg $\ell_i$
  of $\Gamma$, we have $w(\ell_i) \equiv a_i \pmod{r}$.
\item For any edge $e\in E$ corresponding to two half-edges $h, h'\in H$, we
  have $w(h) + w(h') \equiv 0 \pmod{r}$.
\item For any vertex $v \in V$, we have
  $\sum_{h\in p^{-1}(v)} w(h) \equiv -x_v \pmod{r}$.
\end{enumerate}
Motivated by three operations from $\Gamma'$ to $\Gamma$ as above mentioned , we define
three different tautological classes as follows: 
\begin{enumerate}[label=(\roman*)]
\item $\widetilde{\mathcal{D}}_{\Gamma}^r\{1,v^*\}$ to be the class
\begin{equation*}
  \frac 1{r^{h^1(\Gamma)}} \sum_{\substack{w \text{ weighting }\\ \text{mod } r \text{ on }(\Gamma,x)}} [\Gamma, \gamma_w\{1,v^*\}] \in \mathcal{S}_{g,n},
\end{equation*}
where
\begin{equation*}
  \gamma_w\{1,v^*\} =
\left(\sum_{h\in H:p(h)=v^*}\frac{e^{\frac{1}{2} a_h^2\psi_h}-1}{\psi_{h}}\prod_{i=1:i\neq h}^{n} e^{\frac{1}{2} a_i^2\psi_i}\right)
\prod_{(h,h')\in E} \frac{1 - e^{-\frac 12 w(h)w(h')(\psi_h + \psi_{h'})}}{\psi_h + \psi_{h'}}.
\end{equation*}
\item $\widetilde{\mathcal{D}}_{\Gamma}^r\{2,i\}$ to be the class
\begin{equation*}
  \frac 1{r^{h^1(\Gamma)}} \sum_{\substack{w \text{ weighting }\\ \text{mod } r \text{ on }(\Gamma,x)}} [\Gamma, \gamma_w\{2,i\}] \in \mathcal{S}_{g,n},
\end{equation*}
where
\begin{equation*}
  \gamma_w\{2,i\} =\frac{1-e^{\frac{1}{2}(a_i+x_{v^*})^2\psi_i}}{\psi_i}\prod_{j=1: j\neq i}^n e^{\frac 12 a_j^2\psi_j} \prod_{(h,h')\in E} \frac{1 - e^{-\frac 12 w(h)w(h')(\psi_h + \psi_{h'})}}{\psi_h + \psi_{h'}}.
\end{equation*}
\item $\widetilde{\mathcal{D}}_{\Gamma}^r\{3,e\}$ to be the class
\begin{equation*}
  \frac 1{r^{h^1(\Gamma)}} \sum_{\substack{w \text{ weighting }\\ \text{mod } r \text{ on }(\Gamma,x)}} [\Gamma, \gamma_w\{3,e\}] \in \mathcal{S}_{g,n},
\end{equation*}
where
\begin{align*}
\gamma_w\{3,e\}=&\prod_{j=1}^n e^{\frac 12 a_j^2\psi_j} \cdot \frac{1 - e^{-\frac{1}{2} w(h(e))(r-w(h(e)))(\psi_{h(e)} )}}{\psi_{h(e)}}
\frac{1 - e^{-\frac{1}{2}(r-w(h'(e)))w(h'(e)) ( \psi_{h'(e)})}}{ \psi_{h'(e)}}
\\&\cdot\prod_{(h,h')\in E-\{e\}} \frac{1 - e^{-\frac 12 w(h)w(h')(\psi_h + \psi_{h'})}}{\psi_h + \psi_{h'}}    
\end{align*}
where on the special edge $e$, we modify condition (2) to be  
\[w(h(e))+w(h'(e))\equiv x_{v^*}\mod r.\]
The condition (2) is unchanged on the other edges $E-\{e\}$. 
\end{enumerate}

These three kinds of  classes are all polynomials for $r \gg 0$, and we will use
$\mathcal{D}_{\Gamma}^{\{1,v^*\}}$, $\mathcal{D}_{\Gamma}^{\{2,i\}}$ and $\mathcal{D}_{\Gamma}^{\{3,e\}}$  to denote the constant part of the resulting
polynomials respectively.
These classes $\mathcal{D}^{\{1,v^*\}}_{\Gamma}$, $\mathcal{D}^{\{2,i\}}_{\Gamma}$ and $\mathcal{D}^{\{3,e\}}_{\Gamma}$ are all polynomials in the variables
$a_2, \dotsc, a_n$ and $x_v$ for $v \in V\cup\{v^*\}$, and so we will
denote them respectively by
\begin{equation*}
  \mathcal{D}^{\{1,v^*\}}_{\Gamma}(a_2, \dotsc, a_n, \{x_v\}),\,\, \mathcal{D}^{\{2,i\}}_{\Gamma}(a_2, \dotsc, a_n, \{x_v\}),\,\, \mathcal{D}^{\{3,e\}}_{\Gamma}(a_2, \dotsc, a_n, \{x_v\}).
\end{equation*}
Furthermore, let
\begin{equation*}
  \mathcal{D}^{\{1,v^*\}, d}_{\Gamma}(a_2, \dotsc, a_n, \{x_v\}),\,\, \mathcal{D}^{\{2,i\}, d}_{\Gamma}(a_2, \dotsc, a_n, \{x_v\}),\,\, \mathcal{D}^{\{3,e\}, d}_{\Gamma}(a_2, \dotsc, a_n, \{x_v\}),
\end{equation*}
be the corresponding degree $d$ components.


\begin{proposition}
\label{prop:omega_g,m-formula-pi-pushforward}
For any monomial $M$ of degree $D\leq 2g+1$, the push-forward of $\Omega_{g,M}$ under $\pi$ is given by
\begin{small}
\begin{align*}
&\pi_* \Omega_{g, M}
\\=&
\sum_{\Gamma\in G_{g,n}}
\Bigg(\sum_{v^*\in V(\Gamma)}\sum_{\substack{m\colon V(\Gamma) \to \mathbb Z_{\ge 0} \\ \sum_v m(v) = 2g + 1 - D}}
    \frac{(m(v^*) + 1)(2g + 1 - D)!}{|\Aut(\Gamma)|}
    \frac{(2g(v^*) - 2 + n(v^*) + m(v^*))!}{(2g(v^*) - 2 + n(v^*))!}
    \\&\hspace{130pt}\cdot\prod_{v \in V(\Gamma)-\{v^*\}} \frac{(2g(v) - 3 + n(v) + m(v))!}{(2g(v) - 3 + n(v))!}C^{\{1,v^*\}}_{\Gamma,M,m}
   \\&\hspace{35pt}+\sum_{i=1}^{n}\sum_{\substack{m\colon V(\Gamma)\sqcup \{v^*\}\rightarrow\mathbb{Z}_{\geq0}\\ \sum_{v}m(v)=2g+1-D}}\frac{(m(v^*) + 1)!(2g + 1 - D)!}{|\Aut(\Gamma)|}
    \prod_{v \in V(\Gamma)} \frac{(2g(v) - 3 + n(v) + m(v))!}{(2g(v) - 3 + n(v))!}C^{\{2,i\}}_{\Gamma,M,m}
    \\&\hspace{35pt}+\sum_{e\in E(\Gamma)}\sum_{\substack{m\colon V(\Gamma)\sqcup \{v^*\}\rightarrow\mathbb{Z}_{\geq0}\\ \sum_{v}m(v)=2g+1-D}}\frac{(m(v^*) + 1)!(2g + 1 - D)!}{|\Aut(\Gamma)|}
    \prod_{v \in V(\Gamma)} \frac{(2g(v) - 3 + n(v) + m(v))!}{(2g(v) - 3 + n(v))!}C^{\{3,e\}}_{\Gamma,M,m}\Bigg)
\end{align*} 
\end{small}
where
\begin{align*}
  C^{\{1,v^*\}}_{\Gamma, M, m} &= [\mathcal{D}^{\{1,v^*\}, g+1}_{\Gamma}(a_2, \dotsc, a_n, \{x_v\})]_{M \cdot x_{v^*} \prod_v x_v^{m(v)}}, \\
  C^{\{2,i\}}_{\Gamma, M, m} &= [\mathcal{D}^{\{2,i\}, g+1}_{\Gamma}(a_2, \dotsc, a_n, \{x_v\})]_{M \cdot x_{v^*} \prod_v x_v^{m(v)}}, \\
  C^{\{3,e\}}_{\Gamma, M, m} &= [\mathcal{D}^{\{3,e\}, g+1}_{\Gamma}(a_2, \dotsc, a_n, \{x_v\})]_{M \cdot x_{v^*} \prod_v x_v^{m(v)}}.
\end{align*}
\end{proposition}
\begin{proof}
  By Proposition~\ref{prop:omega_g,m-formula}, we have
  \begin{equation*}
    \pi_* \Omega_{g, M} = \sum_{\Gamma'\in G_{g,n+1}} \Cont_{\Gamma'},
  \end{equation*}
  where
  \begin{equation*}
    \Cont_{\Gamma'} = \sum_{\substack{m\colon V(\Gamma') \to \mathbb Z_{\ge 0} \\ \sum_v m(v) = 2g + 1 - D}}
    \frac{(m(v^*) + 1)(2g + 1 - D)!}{|\Aut(\Gamma')|}
    \prod_{v \in V(\Gamma')} \frac{(2g(v) - 3 + n(v) + m(v))!}{(2g(v) - 3 + n(v))!} \cdot \pi_*C_{\Gamma', M, m}.
  \end{equation*}
  As we have discussed, for any $\Gamma \in G_{g, n}$, there are three types of stable graphs $\Gamma' \in G_{g, n+1}$ mapping to $\Gamma$ with respect to $\pi$.
  This allows to divide the above sum into three parts:
  \begin{equation*}
    \pi_* \Omega_{g, M} = \sum_{\Gamma\in G_{g,n}} (\Cont_{\Gamma}^1 + \Cont_{\Gamma}^2 + \Cont_{\Gamma}^3),
  \end{equation*}
  where
\begin{small}
  \begin{align*}
    \Cont_\Gamma^1 =
&
\sum_{\Gamma'\mapsto^{(1)}\Gamma}\sum_{\substack{m\colon V(\Gamma') \to \mathbb Z_{\ge 0} \\ \sum_v m(v) = 2g + 1 - D}}
    \frac{(m(v^*) + 1)(2g + 1 - D)!}{|\Aut(\Gamma')|}
    \prod_{v \in V(\Gamma')} \frac{(2g(v) - 3 + n(v) + m(v))!}{(2g(v) - 3 + n(v))!} \cdot \pi_*C_{\Gamma', M, m}
\\=&
\sum_{v^*\in V(\Gamma)}\sum_{\substack{m\colon V(\Gamma) \to \mathbb Z_{\ge 0} \\ \sum_v m(v) = 2g + 1 - D}}
    \frac{(m(v^*) + 1)(2g + 1 - D)!}{|\Aut(\Gamma)|}
   \frac{(2g(v^*) - 2 + n(v^*) + m(v^*))!}{(2g(v^*) - 2 + n(v^*))!} 
\\&\hspace{120pt}\cdot\prod_{v \in V(\Gamma)-\{v^*\}} \frac{(2g(v) - 3 + n(v) + m(v))!}{(2g(v) - 3 + n(v))!}C^{\{1,v^*\}}_{\Gamma,M,m},
    \end{align*}    
\end{small}
\begin{small}
  \begin{align*}
    \Cont_\Gamma^2 =
&
\sum_{\Gamma'\mapsto^{(2)}\Gamma}\sum_{\substack{m\colon V(\Gamma') \to \mathbb Z_{\ge 0} \\ \sum_v m(v) = 2g + 1 - D}}
    \frac{(m(v^*) + 1)(2g + 1 - D)!}{|\Aut(\Gamma')|}
    \prod_{v \in V(\Gamma')} \frac{(2g(v) - 3 + n(v) + m(v))!}{(2g(v) - 3 + n(v))!} \cdot \pi_*C_{\Gamma', M, m}
\\=&
\sum_{i=1}^{n}\sum_{\substack{m\colon V(\Gamma)\sqcup \{v^*\}\rightarrow\mathbb{Z}_{\geq0}\\ \sum_{v}m(v)=2g+1-D}}\frac{(m(v^*) + 1)(2g + 1 - D)!}{|\Aut(\Gamma)|}
 \frac{m(v^*)!}{0!}   
\\&\hspace{120pt}\cdot\prod_{v \in V(\Gamma)} \frac{(2g(v) - 3 + n(v) + m(v))!}{(2g(v) - 3 + n(v))!}C^{\{2,i\}}_{\Gamma,M,m},
\end{align*}    
\end{small}
and
\begin{small}
  \begin{align*}
    \Cont_\Gamma^3 =
&
\sum_{\Gamma'\mapsto^{(3)} \Gamma}\sum_{\substack{m\colon V(\Gamma') \to \mathbb Z_{\ge 0} \\ \sum_v m(v) = 2g + 1 - D}}
    \frac{(m(v^*) + 1)(2g + 1 - D)!}{|\Aut(\Gamma')|}
    \prod_{v \in V(\Gamma')} \frac{(2g(v) - 3 + n(v) + m(v))!}{(2g(v) - 3 + n(v))!} \cdot \pi_*C_{\Gamma', M, m}
\\=&
    \sum_{e\in E(\Gamma)}\sum_{\substack{m\colon V\sqcup \{v^*\}\rightarrow\mathbb{Z}_{\geq0}\\ \sum_{v}m(v)=2g+1-D}}\frac{(m(v^*) + 1)(2g + 1 - D)!}{|\Aut(\Gamma)|}
    \frac{m(v^*)!}{0!}\\&\hspace{120pt}\cdot\prod_{v \in V(\Gamma)} \frac{(2g(v) - 3 + n(v) + m(v))!}{(2g(v) - 3 + n(v))!}C^{\{3,e\}}_{\Gamma,M,m}
\end{align*}    
\end{small}
where $\Gamma'\mapsto^{(i)}\Gamma$ stands for the type $(i)$ graphs $\Gamma' \in G_{g, n+1}$ mapping to $\Gamma$ under $\pi$ for $i=1,2,3$.
\end{proof}

\section{The coefficient of the bouquet class in topological recursion relations}
\label{sec:coeff-bou}

In this section, we compute the coefficient of the bouquet class in the degree-$g$ topological recursion relations on $\overline{\mathcal{M}}_{g,n}$.
First, we compute contributions of stable graphs to $\pi_{*}\Omega_{g,n}$, in particular, the trivial graph $\Gamma_{g,[n]}$ of genus $g$ with $n$ markings, and the dual graph $\Gammaloop_{g-1, 1, [n]}$ with a single genus $g-1$ vertex, a loop and all of the markings.
Second, we establish a crucial identity between these two contributions.
From there, the algorithm  \eqref{eq:cancellation} determines the coefficient of the bouquet class inductively, and the identities will lead the proof of Theorem~\ref{thm:coeff-bouquet}.

\subsection{Contributions of stable graphs to $\pi_* \Omega_{g, M}$}

Consider the trivial graph $\Gamma_{g,[n]}$ of genus $g$ with $n$ markings, and the dual graph $\Gammaloop_{g-1, 1, [n]}$ with a single genus $g-1$ vertex, a loop and all of the markings.
Furthermore, fix a monomial $M = \psi_2^{b_2} \dots \psi_n^{b_n}$ of degree $D$.
For each of these graphs, we will consider the corresponding contribution to $\pi_* \Omega_{g, M}$ in Proposition~\ref{prop:omega_g,m-formula-pi-pushforward}, and simplify it.

The contribution of the trivial graph $\Gamma_{g,[n]}$ is
\begin{align*}
 &\Cont_{\Gamma_{g,[n]}}\pi_*\Omega_{g,M}\\= & \frac{(2g + 2 - D)!(4g - 1 + n - D)!}{(2g - 2 + n)!} \cdot [\mathcal{D}^{\{1,v^*\}, g+1}_{\Gamma_{g, [n]}}(a_2, \dotsc, a_n, x_v)]_{M \cdot x_v^{2g + 2 - D}} \\
  &+\sum_{i=1}^{n}\sum_{m + m^* = 2g + 1 - D}\frac{(m^* + 1)!(2g + 1 - D)!(2g - 3 + n + m)!}{(2g - 3 + n)!}
\\&\hspace{100pt}\cdot  [\mathcal{D}^{\{2,i\}, g+1}_{\Gamma_{g, [n]}}(a_2, \dotsc, a_n, x_v, x_{v^*})]_{M \cdot x_v^m x_{v^*}^{m^* + 1}}
\end{align*}
where $v$ is the single vertex of $\Gamma_{g,[n]}$.
Note that $a_1 + x_{v^*} = -\sum_{i = 2}^n a_i - x_v$, and hence does not depend on $x_{v^*}$.
Therefore, the $i = 1$-summand in the above expression vanishes.
Unfolding the definition of $\mathcal{D}^{\{2,i\}, g+1}_{\Gamma_{g, [n]}}$, we expand the contribution to
\begin{align}\label{eqn:cont-g-[n]}
\Cont_{\Gamma_{g,[n]}}\pi_*\Omega_{g,M}=&\frac{(2g + 2 - D)!(4g - 1 + n - D)!}{(2g - 2 + n)!} \cdot 
    \left[\Gamma_{g, [n]}, \pi_* \prod_{i=1}^{n} e^{\frac{1}{2} a_i^2\psi_i}\right]^{g + 1}_{M \cdot x_v^{2g + 2 - D}} \nonumber
    \\&+ \sum_{i=2}^{n}\sum_{m + m^* = 2g + 1 - D}\frac{(m^* + 1)!(2g + 1 - D)!(2g - 3 + n + m)!}{(2g - 3 + n)!}
\nonumber\\&\hspace{70pt}\cdot\left[\Gamma_{g, [n]}, \frac{1-e^{\frac{1}{2}(a_i+x_{v^*})^2\psi_i}}{\psi_i}\prod_{j=1: j\neq i}^n e^{\frac 12 a_j^2\psi_j}\right]^{g + 1}_{M \cdot x_v^m x_{v^*}^{m^* + 1}}.
\end{align}

We now pass to the contribution of the dual graph $\Gammaloop_{g-1, 1, [n]}$.
It is given by
\begin{align*}
  &\Cont_{\Gammaloop_{g-1,1,[n]}}\pi_*\Omega_{g,M}\\=& \frac{(2g + 2 - D)!(4g - 1 + n - D)!}{2 \cdot (2g - 2 + n)!} \cdot [\mathcal{D}^{\{1,v^*\}, g+1}_{\Gammaloop_{g-1, 1, [n]}}(a_2, \dotsc, a_n, x_v)]_{M \cdot x_v^{2g + 2 - D}} \\
  &+\sum_{i=1}^{n}\sum_{m + m^* = 2g + 1 - D}\frac{(m^* + 1)!(2g + 1 - D)!(2g - 3 + n + m)!}{2 \cdot (2g - 3 + n)!}
\\&\hspace{100pt}\cdot  [\mathcal{D}^{\{2,i\}, g+1}_{\Gammaloop_{g-1, 1, [n]}}(a_2, \dotsc, a_n, x_v, x_{v^*})]_{M \cdot x_v^m x_{v^*}^{m^* + 1}} \\
  &+ \sum_{m + m^* = 2g + 1 - D} \frac{(m^* + 1)!(2g + 1 - D)!(2g - 3 + n + m)!}{2 \cdot (2g - 3 + n)!} \\&\hspace{80pt}\cdot[\mathcal{D}^{\{3,e\}, g+1}_{\Gammaloop_{g-1, 1, [n]}}(a_2, \dotsc, a_n, x_v, x_{v^*})]_{M \cdot x_{v^*}^{m^* + 1} x_v^{m}}.
\end{align*}
Note that the first two lines vanish by Remark~\ref{rem:contr-asimple-loop}. Unfolding the definition of $\mathcal{D}^{\{3,e\}, g+1}_{\Gammaloop_{g-1, 1, [n]}}$, we expand the remaining contribution to 
\begin{align}\label{eqn:cont-g-1-[n]-0}
&\Cont_{\Gammaloop_{g-1,1,[n]}}\pi_*\Omega_{g,M}
\nonumber\\=&\sum_{m+m^*=2g+1-D}
\frac{(m^*+1)!(2g+1-D)!(2g-3+n+m)!}{2(2g-3+n)!}
\nonumber\\&\hspace{50pt}\cdot\Bigg[\xi_{\Gammaloop_{g-1,1,[n]}},\Coeff_{r^0}\frac{1}{r}\sum_{\substack{w_1,w_2=0\\w_1+w_2\equiv x_{v^*} \mod r}}^{r-1}\Bigg[\exp\Big(\frac{1}{2}(\sum_{j=2}^{n}a_j+x_v+x_{v^*})^2\psi_1+\frac{1}{2}\sum_{j=2}^{n}a_j^2\psi_j\Big)
\nonumber\\&\hspace{140pt}\cdot\frac{1-e^{\frac{1}{2}w_1^2\psi_{h(e)}}}{\psi_{h(e)}} \frac{1-e^{\frac{1}{2}w_2^2\psi_{h'(e)}}}{\psi_{h'(e)}}\Bigg]_{M\cdot x_v^{m}x_{v^*}^{m^*+1}}\Bigg].    
\end{align}
To simplify further, we use the following:
\begin{lemma}
  \label{lema:lim-w1-w2}
  For any $d_1,d_2\geq0$, the constant term of the following polynomial in $r$ for sufficiently large $r$ satisfies
\begin{align*}
\Coeff_{r^{0}}
\left[\frac{1}{r}\sum_{\substack{w_1,w_2=0\\
w_1+w_2\equiv x \mod r}}^{r-1}w_1^{2d_1}
	w_2^{2d_2}  \right]
	=-\frac{1}{\binom{2d_1+2d_2}{2d_1}}x^{2d_1+2d_2}+O(x^{< 2d_1+2d_2}).
\end{align*}
\end{lemma}
\begin{proof}
We need to consider the following sum:
\begin{align*}
\sum_{w=0}^{x}w^{2d_1}(x-w)^{2d_2}  +\sum_{w=x+1}^{r-1}w^{2d_1}(r+x-w)^{2d_2}.  
\end{align*}
The power sums are polynomials in $x$ and $r$. We are only interested in the
coefficient of $r\cdot x^{2d_1+2d_2}$.
We rewrite the expression as
\begin{align*}
\sum_{w=0}^{x}w^{2d_1}((x-w)^{2d_2}-(r+x-w)^{2d_2})  +\sum_{w=0}^{r-1}w^{2d_1}(r+x-w)^{2d_2}.  
\end{align*}
We do not need to consider the second sum since its degree in $x$ is $2d_2 < 2d_1 + 2d_2$.
The linear term in $r$ of the first sum is
\begin{align*}
-2r d_2\sum_{w=0}^{x}w^{2d_1}(x-w)^{2d_2-1}.    
\end{align*}
Then the coefficient of $r x^{2d_1+2d_2}$ can be computed by
\begin{align*}
-2d_2 x^{-2d_1-2d_2}
\int_{0}^{x}w^{2d_1}(x-w)^{2d_2-1} dw.
\end{align*}
By repeated integration by parts we can finally evaluate the coefficient to be:
\begin{align*}
-\frac{1}{\binom{2d_1+2d_2}{2d_1}}   .
\end{align*}
\end{proof}

Then by using Lemma~\ref{lema:lim-w1-w2} and degree considerations, equation~\eqref{eqn:cont-g-1-[n]-0} becomes:
\begin{align}\label{eqn:cont-ga-g-1[n]0}
&\Cont_{\Gammaloop_{g-1,1,[n]}}\pi_*\Omega_{g,M}
\nonumber\\=&\sum_{m+m^*=2g+1-D}
\frac{(m^*+1)!(2g+1-D)!(2g-3+n+m)!}{2(2g-3+n)!}
\nonumber\\&\hspace{50pt}\cdot\Bigg[\xi_{\Gammaloop_{g-1,1,[n]}},\Coeff_{r^0}\frac{1}{r}\sum_{\substack{w_1,w_2=0\\w_1+w_2\equiv x_{v^*} \mod r}}^{r-1}\Bigg[\exp\Big(\frac{1}{2}(\sum_{j=2}^{n}a_j+x_v+x_{v^*})^2\psi_1+\frac{1}{2}\sum_{j=2}^{n}a_j^2\psi_j\Big)
\nonumber\\&\hspace{140pt}\cdot\sum_{d_1,d_2=0}^{\infty}\frac{w_1^{2d_1+2}w_2^{2d_2+2}\psi^{d_1}(\psi')^{d_2}}{2^{d_1+d_2+2}(d_1+1)!(d_2+1)!}\Bigg]_{M\cdot x_v^{m}x_{v^*}^{m^*+1}}\Bigg]
\nonumber\\=&-\sum_{m+m^*=2g+1-D}
\frac{(m^*+1)!(2g+1-D)!(2g-3+n+m)!}{2(2g-3+n)!}
\nonumber\\&\hspace{50pt}\cdot\Bigg[\xi_{\Gammaloop_{g-1,1,[n]}},\Bigg[\exp\Big(\frac{1}{2}(\sum_{j=2}^{n}a_j+x_v+x_{v^*})^2\psi_1+\frac{1}{2}\sum_{j=2}^{n}a_j^2\psi_j\Big)
\nonumber\\&\hspace{140pt}\cdot\sum_{d_1,d_2=0}^{\infty}\frac{(2d_1+1)!!(2d_2+1)!!(x_{v^*})^{2d_1+2d_2+4}\psi^{d_1}(\psi')^{d_2}}{(2d_1+2d_2+4)!}\Bigg]_{M\cdot x_v^{m}x_{v^*}^{m^*+1}}\Bigg].
\end{align}

\subsection{Useful identities}

We collect several useful lemmas, which will be used repeatedly in the proof of Theorem~\ref{thm:coeff-bouquet}.
\begin{lemma}[Pascal's rule]\label{lem:Pascal-rule}For positive integers $n,k$
 with $1\leq k\leq n$:
\begin{align*}
\binom{n}{k-1}+\binom{n}{k}=\binom{n+1}{k}.    
\end{align*}
\end{lemma}

\begin{lemma}\label{lem:cosh-e-indep}For any non-negative integers $g,D,e$ with 
$2g - D - 2e \geq 0$, we have \begin{equation*}
    (2g - D - 2e)! \left[\cosh(a + x)\right]_{x^{2g - D - 2e}}
    = (2g + 2 - D)! \left[\cosh(a + x)\right]_{x^{2g + 2 - D}}.
  \end{equation*}
\end{lemma}
\begin{proof}
  If $D$ is even, then
  \begin{align*}
    &(2g - D - 2e)! \left[\cosh(a + x)\right]_{x^{2g - D - 2e}} = (2g - D - 2e)! \left[\cosh(a) \cosh(x)\right]_{x^{2g - D - 2e}} \\
    = &\cosh(a) =(2g + 2 - D)! \left[\cosh(a + x)\right]_{x^{2g + 2 - D}}.
  \end{align*}
If $D$ is odd, then
  \begin{align*}
    &(2g - D - 2e)! \left[\cosh(a + x)\right]_{x^{2g - D - 2e}} = (2g - D - 2e)! \left[\sinh(a) \sinh(x)\right]_{x^{2g - D - 2e}} \\
    = &\sinh(a) =(2g + 2 - D)! \left[\cosh(a + x)\right]_{x^{2g + 2 - D}}.
  \end{align*}
\end{proof}

We will use the following to simplify some of the sums over $m$ and $m^*$ appearing in the contributions.
\begin{lemma}
  \label{lem:sum-mm}For any non-negative integers $m,m^*$ and $e$, we have
  \begin{align*}
    \sum_{m + m^* = 2g + 1 - D} &\frac{(m^* + 1)!(2g + 1 - D)!(2g - 3 + n + m)!}{(2g - 3 + n)!} \binom{2g + 2 - D - e}{m^* + 1 - e} \\
    = &
        \begin{cases}
          \frac{(2g + 2 - D)!(4g - 1 + n - D)!}{(2g - 2 + n)!} & \text{if } e = 0, \\
          \frac{(2g + 1 - D)! (4g + n - D)! e!}{(2g - 2 + n + e)!} & \text{if } e \ge 1.
        \end{cases}
  \end{align*}
\end{lemma}
\begin{proof}
  We compute
  \begin{align*}
    &\sum_{m + m^* = 2g + 1 - D}\frac{(m^* + 1)!(2g + 1 - D)!(2g - 3 + n + m)!}{(2g - 3 + n)!} \binom{2g + 2 - D - e}{m^* + 1 - e} \\
    = &\sum_{m + m^* = 2g + 1 - D}\frac{(m^* + 1)! (2g + 1 - D)! (2g - 3 + n + m)! (2g + 2 - D - e)!}{(2g - 3 + n)! m! (m^* + 1 - e)!} \\
    = &\sum_{m + m^* = 2g + 1 - D} \binom{2g - 3 + n + m}{2g - 3 + n} \binom{m^* + 1}{e} \cdot e! (2g + 1 - D)! (2g + 2 - D - e)!.
  \end{align*}
Then this lemma follows from the following combinatorial identity:  any $a,b,c\geq0$,
    \begin{align*}
    \sum_{n_1+n_2=a}\binom{b+n_1}{b}\binom{n_2+1}{c}=
    \begin{cases}
    \binom{a+b+2}{b+c+1},    \quad c\geq1 
    \\
    \binom{a+b+1}{b+c+1},    \quad c=0
    .\end{cases}\end{align*}  
\end{proof}

\subsection{A strengthening of Theorem~\ref{thm:coeff-bouquet}}

We formulate a strengthening of both Conjecture~\ref{conj:bouquet-unique} and Theorem~\ref{thm:coeff-bouquet} using two new ingredients:
\begin{itemize}
\item Define a linear form $\mathcal F$ on the space of polynomials of $\psi$ classes via
  \begin{equation*}
    \mathcal F(\prod_{i=1}^{n}\psi_i^{k_i}) = \prod_{i=1}^{n}\frac 1{(2k_i + 1)!!}.
  \end{equation*}
\item For any $h \le g$, we let $\Gammaloop_{h, g-h, [n]}$ be the graph with a single vertex of genus $h$ containing all markings and $g-h$ loops.
  We call these \emph{bouquet type} graphs.
  As a special case $\Gammaloop_{g, 0, [n]} = \Gamma_{g, [n]}$
\end{itemize}
\begin{conjecture}
  \label{conj:bouquet-unique-general}
  Consider an arbitrary topological recursion relation
  \begin{equation*}
    \sum_{\Gamma} {\xi_{\Gamma}}_{*}(c_\Gamma) = 0,
  \end{equation*}
  where each $c_\Gamma$ is a polynomial in $\psi$-classes.
  We then have the identity between the coefficients of bouquet type graphs
  \begin{equation}
    \label{eq:bouquet-identity}
    \sum_{h = 0}^g 8^{-h} \mathcal F(c_{\Gammaloop_{h, g-h, [n]}}) = 0
  \end{equation}
\end{conjecture}
\begin{theorem}
  \label{thm:coeff-bouquet-general}
  Conjecture~\ref{conj:bouquet-unique-general} holds for the topological recursion relations in Corollary~\ref{cor:deg-d-TRR-uniform}.
\end{theorem}
Below in Section~\ref{ss:proof-bouquet}, we will discuss the proof of this theorem.

We now note that Theorem~\ref{thm:coeff-bouquet} is a consequence of Theorem~\ref{thm:coeff-bouquet-general}:
Corollary~\ref{cor:deg-d-TRR-uniform} yields topological recursion relations of the same shape as in Theorem~\ref{thm:coeff-bouquet}, and from there Theorem~\ref{thm:coeff-bouquet-general} implies the correct coefficient of the bouquet class.

\subsection{Applying $\mathcal F$}

In this subsection, we collect useful identities related to applying the operator $\mathcal F$.
To start, we note the identities:
\begin{equation*}
  \mathcal F e^{\frac{1}{2} a^2\psi}
  = \sum_{d = 0}^\infty \frac{a^{2d}}{(d)! 2^d (2d + 1)!!}
  = \sum_{d = 0}^\infty \frac{a^{2d}}{(2d + 1)!}
  = \frac{\sinh(a)}{a},
\end{equation*}
\begin{equation*}
  \mathcal F \frac{e^{\frac{1}{2} a^2\psi} - 1}{\psi}
  = \sum_{d = 1}^\infty \frac{a^{2d}}{(d)! 2^d (2d - 1)!!}
  = \sum_{d = 1}^\infty \frac{a^{2d}}{(2d)!}
  = \cosh(a) - 1.
\end{equation*}

We now apply $\mathcal F$ to the contributions to $\pi_* \Omega_{g, M}$.
For the trivial graph $\Gamma_{g,[n]}$, using \eqref{eqn:cont-g-[n]}, we obtain
\begin{align*}
  &\mathcal{F}\Cont_{\Gamma_{g,[n]}}\pi_*\Omega_{g,M}\\=&\frac{(2g + 2 - D)!(4g - 1 + n - D)!}{(2g - 2 + n)!} \cdot 
    \left[\sum_{i = 1}^n \left(\sum_{d = 1}^\infty \frac{a_i^{2d}}{(2d)!}\right) \prod_{j = 1: j \neq i}^n \left(\sum_{d = 0}^\infty \frac{a_j^{2d}}{(2d + 1)!}\right) \right]_{M \cdot x_v^{2g + 2 - D}} \\
  -& \sum_{i=2}^{n}\sum_{m + m^* = 2g + 1 - D}\frac{(m^* + 1)!(2g + 1 - D)!(2g - 3 + n + m)!}{(2g - 3 + n)!} \\
  & \qquad \qquad \qquad \cdot \left[\left(\sum_{d = 1}^\infty \frac{(a_i + x_{v^*})^{2d}}{(2d)!}\right) \prod_{j=1: j\neq i}^n \left(\sum_{d = 0}^\infty \frac{a_j^{2d}}{(2d + 1)!}\right)\right]_{M \cdot x_v^m x_{v^*}^{m^* + 1}} \\
  &:= A + \sum_{i = 2}^n B_i - \sum_{i = 2}^n C_i,
\end{align*}
where $A$ and $B_i$ are the $i = 1$ and $i \ge 2$ summands in the first term, and the $C_i$ are the summands in the second term.
We note that there is a part of $C_i$ which cancels with $B_i$:
\begin{align*}
  &\sum_{m + m^* = 2g + 1 - D}\frac{(m^* + 1)!(2g + 1 - D)!(2g - 3 + n + m)!}{(2g - 3 + n)!} \\
  & \qquad \qquad \qquad \cdot \left[\left(\sum_{d = 1}^\infty \frac{a_i^{2d}}{(2d)!}\right) \prod_{j=1: j\neq i}^n \left(\sum_{d = 0}^\infty \frac{a_j^{2d}}{(2d + 1)!}\right)\right]_{M \cdot x_v^m x_{v^*}^{m^* + 1}} \\
  = &\sum_{m + m^* = 2g + 1 - D}\frac{(m^* + 1)!(2g + 1 - D)!(2g - 3 + n + m)!}{(2g - 3 + n)!} \binom{2g + 2 - D}{m^* + 1} \\
  & \qquad \qquad \qquad \cdot \left[\left(\sum_{d = 1}^\infty \frac{a_i^{2d}}{(2d)!}\right) \left(\sum_{d = 0}^\infty \frac{(a + x_v)^{2d}}{(2d + 1)!}\right) \prod_{j=2: j\neq i}^n \left(\sum_{d = 0}^\infty \frac{a_j^{2d}}{(2d + 1)!}\right)\right]_{M \cdot x_v^{2g + 2 - D}} \\
  = &\frac{(2g + 2 - D)!(4g - 1 + n - D)!}{(2g - 2 + n)!} \\
  & \qquad \qquad \qquad \cdot \left[\left(\sum_{d = 1}^\infty \frac{a_i^{2d}}{(2d)!}\right) \left(\sum_{d = 0}^\infty \frac{(a + x_v)^{2d}}{(2d + 1)!}\right) \prod_{j=2: j\neq i}^n \left(\sum_{d = 0}^\infty \frac{a_j^{2d}}{(2d + 1)!}\right)\right]_{M \cdot x_v^{2g + 2 - D}} \\
  = & B_i,
\end{align*}
where $a:=\sum_{j=2}^{n}a_j$ and we used Lemma~\ref{lem:sum-mm}  in the second equality.

Using Lemma~\ref{lem:sum-mm}, we may also simplify the remaining terms of $C_i$:
\begin{align*}
  &\sum_{m + m^* = 2g + 1 - D}\frac{(m^* + 1)!(2g + 1 - D)!(2g - 3 + n + m)!}{(2g - 3 + n)!} \\
  & \qquad \qquad \qquad \cdot \left[\left(\sum_{d = 1}^\infty \frac{(a_i + x_{v^*})^{2d} - a_i^{2d}}{(2d)!}\right) \prod_{j=1: j\neq i}^n \frac{\sinh(a_j)}{a_j}\right]_{M \cdot x_v^m x_{v^*}^{m^* + 1}} \\
  =&\sum_{e = 1}^\infty \sum_{m + m^* = 2g + 1 - D}\frac{(m^* + 1)!(2g + 1 - D)!(2g - 3 + n + m)!}{(2g - 3 + n)!} \\
  & \qquad \qquad \qquad \cdot \left[\left(\sum_{d = 1}^\infty \frac{a_i^{2d - e}}{(2d - e)! e!}\right) \prod_{j=1: j\neq i}^n \frac{\sinh(a_j)}{a_j}\right]_{M \cdot x_v^m x_{v^*}^{m^* + 1 - e}} \\
  =&\sum_{e = 1}^\infty \sum_{m + m^* = 2g + 1 - D}\frac{(m^* + 1)!(2g + 1 - D)!(2g - 3 + n + m)!}{(2g - 3 + n)!} \binom{2g + 2 - D - e}{m^* + 1 - e} \\
  & \qquad \qquad \qquad \cdot \left[\left(\sum_{d = 1}^\infty \frac{a_i^{2d - e}}{(2d - e)! e!}\right) \frac{\sinh(a + x_v)}{a + x_v} \prod_{j=2: j\neq i}^n \frac{\sinh(a_j)}{a_j}\right]_{M \cdot x_v^{2g + 2 - D - e}} \\
  =&\sum_{e = 1}^\infty \frac{(2g + 1 - D)! (4g + n - D)!}{(2g - 2 + n + e)!} \\
  & \qquad \qquad \qquad \cdot \left[\left(\sum_{d = 1}^\infty \frac{a_i^{2d - e}}{(2d - e)!}\right) \frac{\sinh(a + x_v)}{a + x_v} \prod_{j=2: j\neq i}^n \frac{\sinh(a_j)}{a_j}\right]_{M \cdot x_v^{2g + 2 - D - e}} \\
  =&\sum_{e = 0}^\infty \frac{(2g + 1 - D)! (4g + n - D)!}{(2g - 1 + n + 2e)!}
  \cdot \left[a_i \frac{\sinh(a + x_v)}{a + x_v} \prod_{j=2}^n \frac{\sinh(a_j)}{a_j}\right]_{M \cdot x_v^{2g + 1 - D - 2e}} 
  \\&\hspace{10pt}+\sum_{e = 1}^\infty \frac{(2g + 1 - D)! (4g + n - D)!}{(2g - 2 + n + 2e)!}
  \cdot \left[\cosh(a_i) \frac{\sinh(a + x_v)}{a + x_v} \prod_{j=2: j\neq i}^n \frac{\sinh(a_j)}{a_j}\right]_{M \cdot x_v^{2g + 2 - D - 2e}},
\end{align*}
where in the last equality we divide the sum into two parts according to whether $e$ is odd or even. 

Note that
\begin{align*}
&\mathcal{F} \sum_{d_1,d_2=0}^{\infty}\frac{(2d_1+1)!!(2d_2+1)!!(x_{v^*})^{2d_1+2d_2+4}\psi^{d_1}(\psi')^{d_2}}{(2d_1+2d_2+4)!}
\\=&\sum_{d_1,d_2=0}^{\infty}\frac{(x_{v^*})^{2d_1+2d_2+4}}{(2d_1+2d_2+4)!}
=\sum_{d=1}^{\infty}(d-1)\frac{(x_{v^*})^{2d}}{(2d)!}.   
\end{align*}
Then applying $\mathcal F$ to the contribution \eqref{eqn:cont-ga-g-1[n]0} of the graph $\Gammaloop_{g-1, 1, [n]}$ to $\pi_* \Omega_{g, M}$, we have
\begin{align}\label{eqn:F-cont-ga-g-1[n]0} &\mathcal{F}\Cont_{\Gammaloop_{g-1,1,[n]}}\pi_*\Omega_{g,M}
\nonumber\\= &- \sum_{m + m^* = 2g + 1 - D} \frac{(m^* + 1)!(2g + 1 - D)!(2g - 3 + n + m)!}{2 \cdot (2g - 3 + n)!} 
\nonumber\\&\hspace{50pt}\cdot\left[\prod_{j = 1}^n \left(\sum_{d = 0}^\infty \frac{a_i^{2d}}{(2d + 1)!}\right) \sum_{d = 1}^\infty \frac{d - 1}{(2d)!} x_{v^*}^{2d}\right]_{M \cdot x_{v^*}^{m^* + 1} x_v^{m}}.
\end{align}

\subsection{Proof of Theorem~\ref{thm:coeff-bouquet}}
\label{ss:proof-bouquet}

We now discuss how to use the identities collected in this section to prove Theorem~\ref{thm:coeff-bouquet-general}, which implies Theorem~\ref{thm:coeff-bouquet}.

Because \eqref{eq:bouquet-identity} is linear in the coefficients of the topological recursion relation, it suffices to consider relations of the shape $\pi_{*}\Omega_{g,M}$ for arbitrary monomials $M =\prod_{i=2}^{n}a_i^{b_i}$ of degree $D\leq 2g+1$, as well as for gluing map push-forwards of those relations of lower genus.
We note that the only bouquet type graphs appearing in the relations $\pi_{*}\Omega_{g,M}$ are $\Gamma_{g, [n]}$ and $\Gammaloop_{g-1, 1, [n]}$.
Thus, \eqref{eq:bouquet-identity} holds for such relations by the following identity, whose proof will be discussed below:
\begin{proposition}\label{prop:1-loop-rel-M}
\begin{align*}
\mathcal{F}\left(
\Cont_{\Gamma_{g,[n]}}\pi_{*}\Omega_{g,M}+8\Cont_{\Gammaloop_{g-1,1,[n]}}\pi_{*}\Omega_{g,M}  \right) =0 
\end{align*}
for any monomial $M=\prod_{i=2}^{n}a_i^{b_i}$ of degree $D\leq 2g+1$.
\end{proposition}
It is further not hard to verify that Proposition~\ref{prop:1-loop-rel-M} also implies that the gluing-map push-forwards of the relations $\pi_{*}\Omega_{g,M}$ must satisfy \eqref{eq:bouquet-identity}.
We conclude that given Proposition~\ref{prop:1-loop-rel-M}, all TRRs as in Corollary~\ref{cor:deg-d-TRR-uniform} satisfy \eqref{eq:bouquet-identity}.
This completes the proof of Theorem~\ref{thm:coeff-bouquet-general}, and thus Theorem~\ref{thm:coeff-bouquet}.

\begin{proof}[Proof of Proposition~\ref{prop:1-loop-rel-M}]
By using  Lemma~\ref{lem:sum-mm} and equation~\eqref{eqn:F-cont-ga-g-1[n]0}, we have
\begin{align*}
&8\mathcal{F}\Cont_{\Gamma_{g-1,1,[n]}}\pi_*\Omega_{g,M}
\\=&-4\sum_{e = 1}^\infty \frac{e - 1}{(2e)!} \sum_{m + m^* = 2g + 1 - D} \frac{(m^* + 1)!(2g + 1 - D)!(2g - 3 + n + m)!}{(2g - 3 + n)!} \\
  &\hspace{20pt}\cdot\left[\prod_{j = 1}^n \left(\sum_{d = 0}^\infty \frac{a_i^{2d}}{(2d + 1)!}\right) \right]_{M \cdot x_{v^*}^{m^* + 1 - 2e} x_v^{m}} \\
  =&-4\sum_{e = 1}^\infty \frac{e - 1}{(2e)!} \sum_{m + m^* = 2g + 1 - D} \frac{(m^* + 1)!(2g + 1 - D)!(2g - 3 + n + m)!}{ (2g - 3 + n)!} \binom{2g + 2 - D - 2e}{m^* + 1 - 2e} \\
  &\hspace{20pt}\cdot\left[\left(\sum_{d = 0}^\infty \frac{(a +x_{v^*}+ x_v)^{2d}}{(2d + 1)!}\right) \prod_{j = 2}^n \left(\sum_{d = 0}^\infty \frac{a_i^{2d}}{(2d + 1)!}\right) \right]_{M \cdot x_{v^*}^{m^* + 1 - 2e} x_v^{m}} \\
  = &-4\sum_{e = 1}^\infty \frac{(e - 1)(2g + 1 - D)! (4g + n - D)!}{(2g - 2 + n + 2e)!} \left[\left(\sum_{d = 0}^\infty \frac{(a + x_v)^{2d}}{(2d + 1)!}\right) \prod_{i = 2}^n \left(\sum_{d = 0}^\infty \frac{a_i^{2d}}{(2d + 1)!}\right) \right]_{M \cdot x_v^{2g + 2 - D - 2e}} \\
  = &\sum_{e = 0}^\infty \frac{(-4e)(2g + 1 - D)! (4g + n - D)!}{(2g + n + 2e)!}
  \left[\frac{\sinh(a + x_v)}{a + x_v } \prod_{i = 2}^n \frac{\sinh(a_i)}{a_i} \right]_{M \cdot x_v^{2g - D - 2e}}.
\end{align*}

We may write
\begin{equation*}
  -4e = (2g - 2e + n) \left(1 - \frac{2g + n + 2e}{2g - 2e - D + 1}\right) + \frac{(D + n - 1)(2g + n + 2e)}{2g - 2e - D + 1}.
\end{equation*}
Introduce differential operator
\begin{equation*}
  \Theta = \left(x_v \frac{d}{dx_v} + 1\right) + \sum_{i = 2}^n \left(a_i \frac d{da_i} + 1\right).
\end{equation*}
Then,
\begin{equation*}
  -4e Mx_v^{2g - D - 2e} = \left(\left(1 - \frac{2g + n + 2e}{2g - 2e - D + 1}\right)\Theta + \frac{2g + n + 2e}{2g - 2e - D + 1} \sum_{i = 2}^n \left(a_i \frac d{da_i} + 1\right)\right) Mx_v^{2g - D - 2e}.
\end{equation*}
Note that
\begin{equation*}
  \Theta \left(\frac{\sinh(a + x_v)}{a + x_v } \prod_{i = 2}^n \frac{\sinh(a_i)}{a_i}\right)
  = \cosh(a + x_v) \prod_{i = 2}^n \frac{\sinh(a_i)}{a_i}
  + \sum_{i = 2}^n \frac{\sinh(a + x_v)\cosh(a_i)}{a + x_v } \prod_{\substack{j = 2 \\ j \neq i}}^n \frac{\sinh(a_j)}{a_j}
\end{equation*}
and
\begin{align*}
  &\left(a_i \frac d{da_i} + 1\right) \left(\frac{\sinh(a + x_v)}{a + x_v } \prod_{i = 2}^n \frac{\sinh(a_i)}{a_i}\right) 
  \\
  =& \left(\frac{a_i \cosh(a + x_v)}{a + x_v}
  - \frac{a_i \sinh(a + x_v)}{(a + x_v)^2}\right) \prod_{i = 2}^n \frac{\sinh(a_i)}{a_i}
  + \frac{\sinh(a + x_v)\cosh(a_i)}{a + x_v } \prod_{\substack{j = 2 \\ j \neq i}}^n \frac{\sinh(a_j)}{a_j} \\
  =& \frac d{dx_v} \frac{a_i \sinh(a + x_v)}{a + x_v} \prod_{i = 2}^n \frac{\sinh(a_i)}{a_i}
  + \frac{\sinh(a + x_v)\cosh(a_i)}{a + x_v } \prod_{\substack{j = 2 \\ j \neq i}}^n \frac{\sinh(a_j)}{a_j}.
\end{align*}
So we have
\begin{align*}
  -\sum_{e = 0}^\infty \frac{4e(2g + 1 - D)! (4g + n - D)!}{(2g + n + 2e)!}
  \left[\frac{\sinh(a + x_v)}{a + x_v } \prod_{i = 2}^n \frac{\sinh(a_i)}{a_i} \right]_{M \cdot x_v^{2g - D - 2e}}
  = A' + \sum_{i = 2}^n (B_i' + C_i'),
\end{align*}
where
\begin{align*}
  A' &= \sum_{e = 0}^\infty \frac{(2g + 1 - D)! (4g + n - D)!}{(2g + n + 2e)!} \left(1 - \frac{2g + n + 2e}{2g - 2e - D + 1}\right)
       \left[\cosh(a + x_v) \prod_{i = 2}^n \frac{\sinh(a_i)}{a_i} \right]_{M \cdot x_v^{2g - D - 2e}}, \\
  B_i' &= \sum_{e = 0}^\infty \frac{(2g + 1 - D)! (4g + n - D)!}{(2g + n + 2e)!}
         \left[\frac{\sinh(a + x_v) \cosh(a_i)}{a + x_v} \prod_{\substack{j = 2 \\ j \neq i}}^n \frac{\sinh(a_j)}{a_j} \right]_{M \cdot x_v^{2g - D - 2e}}, \\
  C_i' &= \sum_{e = 0}^\infty \frac{(2g + 1 - D)! (4g + n - D)!}{(2g + n + 2e)!} \frac{2g + n + 2e}{2g - 2e - D + 1}
         \left[\frac d{dx_v} \frac{a_i \sinh(a + x_v)}{a + x_v} \prod_{i = 2}^n \frac{\sinh(a_i)}{a_i} \right]_{M \cdot x_v^{2g - D - 2e}}.
\end{align*}
We may rewrite $C_i'$ slightly
\begin{equation*}
  C_i' = \sum_{e = 0}^\infty \frac{(2g + 1 - D)! (4g + n - D)!}{(2g + n + 2e - 1)!}
         \left[\frac{a_i \sinh(a + x_v)}{a + x_v} \prod_{i = 2}^n \frac{\sinh(a_i)}{a_i} \right]_{M \cdot x_v^{2g - D - 2e + 1}}.
\end{equation*}
Therefore,
\begin{equation*}
  B_i - C_i + B_i' + C_i' = 0.
\end{equation*}
Furthermore,
applying the Lemma~\ref{lem:cosh-e-indep}, we see
\begin{equation*}
  A' = c \cdot \left[\cosh(a + x_v) \prod_{i = 2}^n \frac{\sinh(a_i)}{a_i} \right]_{M \cdot x_v^{2g + 2 - D}},
\end{equation*}
where
\begin{align*}
  c &= \sum_{e = 0}^\infty \frac{(2g + 1 - D)! (4g + n - D)!}{(2g + n + 2e)!} \frac{(2g + 2 - D)!}{(2g - D - 2e)!} \left(1 - \frac{2g + n + 2e}{2g - 2e - D + 1}\right)
\\=&(2g+1-D)!(2g+2-D)!\sum_{e=0}^{\infty}\left(\binom{4g+n-D}{2g+n+2e}-\binom{4g+n-D}{2g+n+2e-1}\right)
    \\=&(2g+1-D)!(2g+2-D)!\sum_{e=0}^{\infty}\left(\binom{4g+n-D-1}{2g+n+2e}-\binom{4g+n-D-1}{2g+n+2e-2}\right)
    \\=&-(2g+1-D)!(2g+2-D)!\binom{4g+n-D-1}{2g+n-2}
    \\=&-\frac{(2g+2-D)!(4g+n-D-1)!}{(2g+n-2)!},
\end{align*}
where in the third equality we used Pascal's rule (see Lemma~\ref{lem:Pascal-rule})  for the two terms 
\begin{align*}
\binom{4g+n-D}{2g+n+2k},\,\,\binom{4g+n-D}{2g+n+2k-1}     
\end{align*}
respectively. 
Therefore $A + A' = 0$.
\end{proof}

\section{The coefficients of rational tail classes in topological recursion relations}
\label{sec:boun-g-2}

In this section, we study applications of Faber's intersection number conjecture to topological recursion relations to yield linear relations satisfied by the rational tail coefficients of topological recursion relations.
We will assume $g \ge 2$ throughout this section.

\subsection{Faber's intersection number conjecture}

In 1993, C. Faber  proposed remarkable conjectures about the structure of
tautological ring $R^*(\mathcal{M}_{g})$ (cf. \cite{faber1999conjectural}).
An important part of Faber's conjectures is the famous Faber intersection number
conjecture, which is the following relation in $R^{g-2}(\mathcal{M}_{g})$, if $\sum_{i=1}^{n}k_i=g-2$, 
\begin{align*}
\pi_*(\psi_1^{k_1+1}\dots \psi_{n}^{k_n+1})=\frac{(2g-3+n)!(2g-1)!!}{(2g-1)!\prod_{i=1}^{n}(2k_i+1)!!} \kappa_{g-2}   
\end{align*}
where $\pi: \mathcal{M}^{rt}_{g,n}\rightarrow\mathcal{M}_{g}$ is the forgetful morphism and $\mathcal{M}^{rt}_{g,n}\subset\overline{\mathcal{M}}_{g,n}$ is the partial compactification of $\mathcal{M}_{g,n}$ by stable nodal
 curves with rational tails.
By now, there are many proofs of Faber's intersection number conjecture (cf.\ \cite{liu2009proof}, \cite{buryak2011new}, \cite{garcia2022curious}). It was pointed out in \cite{liu2009proof} that Faber's intersection number conjecture is equivalent to the following identity of Hodge integrals
\begin{align}\label{eqn:faber-int-conj-hodge}
\int_{\overline{\mathcal{M}}_{g,n}}\psi_1^{k_1}\dots \psi_{n}^{k_n}\lambda_g\lambda_{g-1}=\frac{(2g-3+n)!
|B_{2g}|}{2^{2g-1}(2g)!\prod_{i=1}^{n}(2k_i-1)!!}   
\end{align}
for any $\sum_{i=1}^{n}k_i=g-2+n$.

\subsection{Proof of Theorem~\ref{thm:trr-g-2}}
For $g=1$, the genus 1 topological recursion relation on $\overline{\mathcal{M}}_{1,n}$ is obtained by pullback from the fundamental topological recursion relation $\psi_1=\frac{1}{12}\delta_{irr}\in R^1(\overline{\mathcal{M}}_{1,1})$, where $\delta_{irr}$ is the irreducible boundary divisor.
For $g\geq2, n=2$, assume that the genus $g$ topological recursion relation has the form
\begin{align}\label{eqn:thm-trr-g-2-assume}
\psi_1^{k_1}\psi_2^{k_2}=c_{g,2}\cdot\left(\xi_{\Gamma_{g;0,\{1,2\}}}\right)_*\psi_{h(e)}^{g-1}+\cdots
\end{align}
where $\Gamma_{g;0,\{1,2\}}$ is the stable graph having two vertices connected by one edge, with one genus-$g$ vertex and one genus-0 vertex with markings $\{1,2\}$, and $h(e)$ is the half edge attached to the genus-$g$ vertex.  
Multiplying by $\lambda_{g}\lambda_{g-1}$ on both sides of \eqref{eqn:thm-trr-g-2-assume}, and using properties of Hodge classes, we obtain
\begin{align*}
\int_{\overline{\mathcal{M}}_{g,2}}\psi_1^{k_1}\psi_{2}^{k_2}\lambda_g\lambda_{g-1}=c_{g,2}\int_{\overline{\mathcal{M}}_{g,1}}\psi_1^{g-1}\lambda_g\lambda_{g-1}.   
\end{align*}
By \eqref{eqn:faber-int-conj-hodge}, we have
\begin{align*}
\frac{(2g-1)!
|B_{2g}|}{2^{2g-1}(2g)!\prod_{i=1}^{2}(2k_i-1)!!}=\frac{(2g-2)!
|B_{2g}|}{2^{2g-1}(2g)!(2g-3)!!}  \cdot c_{g,2},  
\end{align*}
thus
\[c_{g,2}=\frac{(2g-1)!!}{(2k_1-1)!!(2k_2-1)!!},\]
as desired.
For $g\geq2, n\geq3$, by Corollary~\ref{cor:deg-d-TRR-uniform}, there exists a degree $g$ topological recursion relation of the form
\begin{align}\label{eqn:deg=g-rational}
\prod_{i=1}^{n}\psi_i^{k_i}=a_0\xi_{*}(\psi_{\bullet}^{g-1})+\sum_{i<j}a_{ij}\xi_{ij*}(\psi^{g-2}_{\bullet})+\cdots, \end{align}
where $\xi:\overline{\mathcal{M}}_{g,\{\bullet\}}\times\overline{\mathcal{M}}_{0,\bar{\bullet}\cup\{1,\dots,n\}}\rightarrow \overline{\mathcal{M}}_{g,n}$  and $\xi_{ij}:\overline{\mathcal{M}}_{g,\{\bullet\}}\times\overline{\mathcal{M}}_{0,\bar{\bullet}\cup\{1,\dots,\hat{i},\dots,\hat{j},\dots,n\}\cup\circ}\times\overline{\mathcal{M}}_{0,\{\bar{\circ},i,j\}}\rightarrow \overline{\mathcal{M}}_{g,n}$ are the gluing maps. The omitted terms consist of other boundary classes with no $\kappa$ classes and no genus-0 vertex decorated by $\psi$ classes.
Multiplying by $\lambda_g\lambda_{g-1}\cdot\xi_{*}({\psi_{\bar{\bullet}}}^{n-3})$ on both sides of equation~\eqref{eqn:deg=g-rational},
note that by product formula~\eqref{eqn:stra-prod} and degree reason,
\begin{align*}
\xi_{*}(\psi_{\bullet}^{g-1})\cdot\xi_{*}(\psi_{\bar{\bullet}}^{n-3})=\xi_{*}(\psi_{\bullet}^{g-1}\psi_{\bar{\bullet}}^{n-3}(-\psi_{\bullet}-\psi_{\bar{\bullet}}))   
\end{align*}
and
\begin{align*}
\xi_{ij*}(\psi_{\bullet}^{g-2})\cdot\xi_{*}(\psi_{\bar{\bullet}}^{n-3})=\xi_{ij*}(\psi_{\bullet}^{g-2}\psi_{\bar{\bullet}}^{n-3}(-\psi_{\bullet}-\psi_{\bar{\bullet}})).    
\end{align*}
Integrating \eqref{eqn:deg=g-rational} over $\overline{\mathcal{M}}_{g,n}$ yields
\begin{align*}
  0=a_0+\sum_{1\leq i<j\leq n}a_{ij}.
\end{align*}
In more detail, by a direct computation, we have
\begin{align*}
\int_{\overline{\mathcal{M}}_{g,n}}\prod_{i=1}^{n}\psi_{i}^{k_i} \cdot \xi_{*}(\psi_{\bar{\bullet}}^{n-3})\lambda_{g}\lambda_{g-1}=\int_{\overline{\mathcal{M}}_{g,\{\bullet\}}}\lambda_{g}\lambda_{g-1} 
\int_{\overline{\mathcal{M}}_{0,\bar{\bullet}\cup\{1,\dots,n\}}}\psi_{\bar{\bullet}}^{n-3}\prod_{i=1}^{n}\psi_{i}^{k_i}=0,
\end{align*}
\begin{align*}
&\int_{\overline{\mathcal{M}}_{g,n}}
\xi_{*}(\psi_{\bullet}^{g-1})\xi_{*}(\psi_{\bar{\bullet}}^{n-3})\lambda_{g}\lambda_{g-1}=\int_{\overline{\mathcal{M}}_{g,n}}
\xi_{*}(\psi_{\bullet}^{g-1}\psi_{\bar{\bullet}}^{n-3}(-\psi_{\bullet}-\psi_{\bar{\bullet}}))\lambda_{g}\lambda_{g-1}
\\=&-\int_{\overline{\mathcal{M}}_{g,\{\bullet\}}}\psi_{\bullet}^{g-1}\lambda_{g}\lambda_{g-1}\int_{\overline{\mathcal{M}}_{0,\bar{\bullet}\cup\{1,\dots,n\}}}\psi_{\bar{\bullet}}^{n-2}
\\=&-\frac{(2g-2)!
|B_{2g}|}{2^{2g-1}(2g)!(2g-3)!!},
\end{align*}
and for every $1 \le i < j \le n$, we have
\begin{align*}
&\int_{\overline{\mathcal{M}}_{g,n}}\xi_{ij*}(\psi_{\bullet}^{g-2})\xi_{*}(\psi_{\bar{\bullet}}^{n-3})\lambda_{g}\lambda_{g-1}
=\int_{\overline{\mathcal{M}}_{g,n}}\xi_{ij*}(\psi_{\bullet}^{g-2}\psi_{\bar{\bullet}}^{n-3}(-\psi_{\bullet}-\psi_{\bar{\bullet}}))\lambda_{g}\lambda_{g-1}
\\=&-\int_{\overline{\mathcal{M}}_{g,\{\bullet\}}}\psi_{\bullet}^{g-1}\lambda_{g}\lambda_{g-1}\int_{\overline{\mathcal{M}}_{0,\bar{\bullet}}\cup\{1,\dots,\hat{i},\dots,\hat{j},\dots,n\}\cup\circ}\psi_{\bullet}^{n-3}
\\=&-\frac{(2g-2)!|B_{2g}|}{2^{2g-1}(2g)!(2g-3)!!}.
\end{align*}
In addition, it is not hard to see 
the intersection of all omitted terms and $\lambda_g\lambda_{g-1}\cdot\xi_{*}({\psi_{\bar{\bullet}}}^{n-3})$ vanishes.
This finishes the proof of Theorem~\ref{thm:trr-g-2}.

\section{Applications to Gromov--Witten theory}
\label{sec:app-gw}

In this section, we disscus some applications of topological recursion relations in Gromov--Witten theory.
Let $X$ be a smooth projective variety.
By definition, the {\it small phase space} is the vector space $H^{*}(X,\mathbb{C})$ and the {\it big phase space} is defined to be
$\mathcal{P}:=\prod_{n=0}^{\infty} H^{*}(X,\mathbb{C})$.
For simplicity, we assume $H^{\text{odd}}(X,\mathbb{C})=0$.
Let $\{\phi_{1},\dots,\phi_{N}\}$ be a fixed basis of $H^{*}(X,\mathbb{C})$, where $\phi_{1}$
is the identity element of the cohomology ring of $M$.
The  corresponding
basis for the $n$-th copy of $H^{*}(X,\mathbb{C})$ in this product is denoted by
$\{ \tau_{n}(\phi_\alpha) \mid \alpha=1, \ldots, N\}$ for $ n \geq 0$.
Let $\{t_{n}^{\alpha}\}$ be the coordinates on $\mathcal{P}$
with respect to the standard basis $\{ \tau_{n}(\phi_\alpha) \mid \alpha=1, \ldots, N, \,\,\, n \geq 0\}$.
 We will also identify the small phase space with
the subspace of $\mathcal{P}$ defined by $t_n^{\alpha}=0$ for $n>0$.
%
%
%
We will identify $\tau_{n}(\phi_{\alpha})$ with the vector field $\frac{\partial}{\partial t_{n}^{\alpha}}$ on the big phase space.
If $n<0$, $\tau_{n}(\phi_{\alpha})$ is understood to be the zero vector field.
We will also write $\tau_{0}(\phi_{\alpha})$ simply as $\phi_{\alpha}$.
We use $\tau_{+}$ and $\tau_{-}$ to denote the operators which shift the level of descendants by $1$, i.e.
$$\tau_{\pm}\Big(\sum_{n,\alpha}f_{n,\alpha}\tau_{n}(\phi_{\alpha})\Big)
=\sum_{n,\alpha}f_{n,\alpha}\tau_{n\pm1}(\phi_{\alpha})$$ where $f_{n,\alpha}$ are functions on the big phase space.

Let $\eta=(\eta_{\alpha\beta})$ be the matrix of the intersection pairing on $H^{*}(X;\mathbb{C})$  in the basis $\{\phi_{1},\dots,\phi_{N}\}$.
We will use
$\eta=(\eta_{\alpha\beta})$ and $\eta^{-1}=(\eta^{\alpha\beta})$ to lower and  raise indices, respectively, for example $\phi^{\alpha}:=\eta^{\alpha\beta}\phi_{\beta}$ for any $\alpha$.
Here we use the Einstein summation convention that repeated indices should be summed over their entire ranges.

Let
\[\langle\tau_{n_{1}}(\phi_{\alpha_{1}})\cdot\cdot\cdot \tau_{n_{k}}(\phi_{\alpha_{k}})\rangle_{g,k,\beta}
:=\int_{[\overline{\mathcal{M}}_{g,k}(X,\beta)]^{vir}}\bigcup_{i=1}^{k}(\Psi_{i}\cup ev_{i}^{*}(\phi_{\alpha_{i}})) \]
be the genus-$g$, degree-$\beta$, descendant Gromov--Witten invariant associated to $\phi_{\alpha_{1}},\dots,\phi_{\alpha_{k}}$ and nonnegative integers $n_{1},\dots,n_{k}$ (cf.\  \cite{li1998virtual}). Here $\overline{\mathcal{M}}_{g,n}(X,\beta)$ is the moduli space of stable maps from genus-$g$, $k$-marked curves to $X$ of degree $\beta\in H_{2}(X;\mathbb{Z})$  and $[\overline{\mathcal{M}}_{g,n}(X,\beta)]^{vir}$ is its virtual fundamental class.
$\Psi_{i}$ is the first Chern class of the tautological line bundle over $\overline{\mathcal{M}}_{g,n}(X,\beta)$ whose geometric fiber is the cotangent space of the domain curve at the $i$-th marked point and  $ev_{i}\colon \overline{\mathcal{M}}_{g,n}(X,\beta)\rightarrow X$ is the $i$-th evaluation map for all $i=1,\dots,k$.
The genus-$g$ generating function is defined to be
\begin{align*}
F_{g}:=\sum_{k\geq0}\sum_{\alpha_{1},\dots,\alpha_{k}}
\sum_{n_{1},\dots,n_{k}}
\frac{1}{k!}t_{n_{1}}^{\alpha_{1}}\cdot\cdot\cdot t_{n_{k}}^{\alpha_{k}}\sum_{\beta}q^{\beta}
\left<\tau_{n_{1}}(\phi_{\alpha_{1}})\cdot\cdot\cdot \tau_{n_{k}}(\phi_{\alpha_{k}})\right>_{g,k,\beta}
\end{align*}
where $q^{\beta}$ belongs to the Novikov ring. Here we set all the undefined terms in $F_g$ to be zero.
This function is understood as a formal power series in the variables $\{t_{n}^{\alpha}\}$ with coefficients in the Novikov ring.

Define a $k$-tensor $\langle\langle\cdot\cdot\cdot\rangle\rangle_g$ by
\begin{align*}
\langle\langle{W_{1}W_{2}\cdot\cdot\cdot W_{k}}\rangle\rangle_g
:=\sum_{m_{1},\alpha_{1},\dots,m_{k},\alpha_{k}}f_{m_{1},\alpha_{1}}^{1}
\cdot\cdot\cdot f_{m_{k},\alpha_{k}}^{k}\frac{\partial^{k}}{\partial t_{m_{1}}^{\alpha_{1}}\cdot\cdot\cdot \partial t_{m_{k}}^{\alpha_{k}}}F_{g}
\end{align*}
for vector fields $W_{i}=\sum_{m,\alpha}f_{m,\alpha}^{i}\frac{\partial}{\partial t_{m}^{\alpha}}$ where $f_{m,\alpha}^{i}$ are functions on the big phase space. 

For any vector fields $W_{1}$ and $W_{2}$ on the big phase space, the {\it quantum product} of $W_{1}$ and $W_{2}$ is defined by
$$W_{1}\bullet W_{2}:=\sum_{\alpha}\langle\langle{W_{1}W_{2}
\phi^{\alpha}}\rangle\rangle_0\phi_{\alpha}.$$

It is well known that topological recursion relations in $R^*(\overline{\mathcal{M}}_{g,n})$ can be translated into universal equations for
Gromov--Witten invariants via the splitting axiom and cotangent line comparison equations.
Define the operator $T$ on the space of vector fields by
\begin{align}\label{eqn:T-operator}
T(W)=\tau_{+}(W)-\sum_{\alpha}\langle\langle W\phi^\alpha\rangle\rangle_0\phi_\alpha.    
\end{align}
The operator is very useful for translating topological recursion relations  into universal equations (cf.\ \cite{liu2006gromov}).
In the process, each marked point corresponds to a vector field, and the cotangent line class corresponds to the operator $T$.
Each node is translated into a pair
of primary vector fields $\phi_\alpha$ and $\phi^\alpha$.
For  the convenience of the reader, we give a proof of the translation rule in Appendix~\ref{subsec:anc-des}.

It is well known that such universal equations play an important role in computing higher genus Gromov--Witten invariants and the study of the famous Virasoro conjecture (cf.\ \cite{liu2002quantum}).
It is also conjectured the collection of all universal equations determines all higher genus Gromov--Witten invariants  in terms of genus-0 Gromov--Witten invariants in the semisimple case (cf.\ \cite{liu2006gromov}), which differs from the approach of Givental (cf.\ \cite{givental2001semisimple}) or Dubrovin-Zhang (cf.\ \cite{dubrovin1998bihamiltonian}).   
Thus it is important to find as explicit as possible formulas for such universal  equations, even though they tend to be complicated in general.

Theorem~\ref{thm:coeff-bouquet} implies:
\begin{corollary}\label{cor:bouquet-correlation-fun}
For any smooth projective variety $X$ and  non-negative integers $\{k_i\}_{i=1}^{n}$ satisfying $\sum_{i=1}^{n}k_i=g$, there exists a universal equation between its Gromov--Witten invariants of the form
\begin{align*}
\langle\langle T^{k_1}(W_1)\ldots T^{k_n}(W_n)\rangle\rangle_g =\frac{1}{8^g \prod_{i=1}^{n}(2k_i+1)!!}\sum_{\alpha_1,\dots ,\alpha_g=1}^{N}\langle\langle W_1\ldots W_n\phi_{\alpha_1}\phi^{\alpha_1}\ldots \phi_{\alpha_g}\phi^{\alpha_g}\rangle\rangle_0+\cdots
\end{align*}
 The omitted terms in the above formula are polynomials of tensors $\{\langle\langle\ldots\rangle\rangle_h: 0\leq h\leq g\}$ which satisfy the following condition: there are no
 genus-$h$ tensors $\langle\langle\ldots\rangle\rangle_h$  with insertions  of operator $T$  of degree $\geq h+\delta_{h}^0$ for
$0\leq h\leq g$.
\end{corollary}

For a function $f$  on the big phase space, we say $f\overset{g}{\approx}0$ if $f$ can be
expressed as a polynomial of $F_0,F_1,\dots,F_{g-1}$ and their derivatives.
Theorem~\ref{thm:trr-g-2} implies the following:
\begin{corollary}\label{cor:two-pt-correlation-fun}
For any smooth projective variety and  any non-negative integers $k_1, k_2$ such that $k_1+k_2=g$, 
 there exists a universal equation between its Gromov--Witten invariants of the form
\begin{align*}
\langle\langle T^{k_1}(W_1)T^{k_2}(W_2)\rangle\rangle_g  -\frac{(2g-1)!!}{(2k_1-1)!! (2k_2-1)!!}\langle\langle T^{g-1}(W_1\bullet W_2)\rangle\rangle_g\overset{g}{\approx}0.  
\end{align*}
\end{corollary}
For $n\geq3$, the number of rational type strata grows rapidly, and it becomes very complicated to compute the coefficients in the degree-$g$ topological recursion relations explicitly.
We do still have the following weaker result on topological recursion relations, which is a corollary from the existence of general degree-$g$ topological recursion relations \eqref{eqn:want} on $\overline{\mathcal{M}}_{g,n}$.
\begin{corollary}
For any smooth projective variety and  any non-negative integers $\{k_i\}_{i=1}^{n}$ satisfying $\sum_{i=1}^{n}k_i\geq g+n-1$, there exists a topological recursion relation between its Gromov--Witten invariants of the form
\begin{align*}
\langle\langle T^{k_1}(W_1)\dots T^{k_n}(W_n)\rangle\rangle_g\overset{g}{\approx}0.    
\end{align*}
Moreover if $\sum_{i=1}^{n}k_i>3g-3+n$, then \begin{align*}
\langle\langle T^{k_1}(W_1)\dots T^{k_n}(W_n)\rangle\rangle_g=0.    
\end{align*}
\end{corollary}
\begin{proof}
  The first equation follows from the fact that if $\sum_{i=1}^{n}k_i\geq g+n-1$, then by degree reason, all coefficients of rational tails in topogical recursion relations $\prod_{i=1}^{n}\psi_i^{k_i}= \text{linear combination of boundary classes}$ vanish.

  For the second equation, note that if $\sum_{i=1}^{n}k_i\geq 3g-3+n$, then $\prod_{i=1}^{n}\psi_i^{k_i}$ vanishes by degree reasons.
\end{proof}


\section{Application to the Hodge class $\lambda_g$}\label{sec:appl-lambda-g}

In \cite{janda2017double}, a nice formula
for the  top Chern class $\lambda_g$ of the Hodge bundle  was found, which is supported on the divisor of curves with a nonseparating node.
More explicitly, we have
\begin{align*}
\lambda_g=\frac{(-1)^g}{2^g} \mathcal{D}_{g,n}^{g}(0,\dots,0) \in R^{g}(\overline{\mathcal{M}}_{g,n})
\end{align*}
where $\mathcal{D}_{g,n}(a_1,\dots,a_n)$ is Pixton's formula for double ramification cycles defined in Section~\ref{subsec: pixton-formula}. 
As an application of Theorem~\ref{thm:coeff-bouquet}, we may compute the coefficient of the bouquet class in this formula for $\lambda_g$:
\begin{proposition}
  \label{prop:lambda_g}
    On $\overline{\mathcal{M}}_{g,n}$, the top Chern class of Hodge bundle may be expressed as
\begin{align*}
\lambda_g=\left[\frac{(-1)^g}{2^g}\frac{t^2e^t}{(e^t-1)^2}\right]_{t^{2g}}
\left(\xi_{\Gammaloop_{0,g,[n]}}\right)_*(1)+\cdots    
\end{align*}
where $\Gammaloop_{0,g,[n]}$ is the stable graph with a single genus-0 vertex equipped with $g$ loops and all markings.
The omitted terms in the above formula satisfy the following condition: there are no
$\kappa$ classes and no genus-$h$ vertices with monomial of $\psi$ classes of degree $\geq h+\delta_{h}^0$ for
$0\leq h\leq g$. $[\quad]_{t^{2g}}$ means taking the coefficient of $t^{2g}$.
\end{proposition}
We state a simple lemma used in the proof of the proposition.
\begin{lemma}\label{lem:kp}
  For any $p\geq1$, we have
  \begin{align*}
  \Coeff_{r^0}\left[\frac{1}{r}\sum_{k=0}^{r-1}  \left(k(r-k)\right)^p\right]
  =(-1)^{p}B_{2p}.
  \end{align*}
\end{lemma}
\begin{proof}
  This is a direct consequence of Faulhaber's formula:
  \begin{align*}
    \sum_{k=1}^{r}k^p
    =\frac{1}{p+1}\sum_{j=0}^{p}(-1)^{j}
    \binom{p+1}{j} B_{j} \cdot r^{p-j+1}.
  \end{align*}
\end{proof}
\begin{proof}[Proof of Proposition~\ref{prop:lambda_g}]
It suffices to compute the coefficients of $\mathcal{D}_{g,n}^g(0,\dots,0)$ corresponding to graphs with a single vertex.
Denote $\Gammaloop_{g-m, m, [n]} \in G_{g,n}$ to be the stable graph with one vertex and $m$ loops.
We start by computing the contribution of the graph $\Gammaloop_{g-1, 1, [n]}$ to $\mathcal{D}_{g,n}^g(0,\dots,0)$
\begin{align*}
&\Cont_{\Gammaloop_{g-1,1,[n]}}\mathcal{D}_{g,n}^g(0,\dots,0)
=\Coeff_{r^0}\left[\sum_{w=0}^{r-1}\frac{1}{2}\frac{1}{r}(\xi_{\Gammaloop_{g-1,1,[n]}})_*\frac{1-\exp(-w(r-w)(\psi_{h}+\psi_{h'}))}{\psi_{h}+\psi_{h'}}\right]_{\deg=g}
\\=&\Coeff_{r^0}\left(\frac{1}{2}\frac{1}{r}(-1)(\xi_{\Gammaloop_{g-1,1,[n]}})_* \frac 1{g!} \sum_{w=0}^{r-1}(w^2)^g(\psi_h+\psi_{h'})^{g-1}\right)
\\=&\frac{-1}{2}\frac{1}{g!}B_{2g}(\xi_{\Gammaloop_{g-1,1,[n]}})_*(\psi_h+\psi_{h'})^{g-1},
\end{align*}
where in the last step, we have used Lemma~\ref{lem:kp}.

By a similar computation, for any $1\leq m\leq g$, we have
\begin{align*}
  \Cont_{\Gammaloop_{g-m,m,[n]}}\mathcal{D}_{g,n}^g(0,\dots,0)
  =\frac{1}{m!\cdot 2^m }(-1)^m\sum_{k_1+\cdots +k_{m}=g}\prod_{i=1}^{m}\frac{B_{2k_i}}{k_i!} \cdot (\xi_{\Gammaloop_{g-m,m,[n]}})_*\prod_{i=1}^{m}(\psi_{h_i}+\psi_{h'_i})^{k_i-1}.
\end{align*}
Combining with Theorem~\ref{thm:coeff-bouquet}, we have 
\begin{align*}
&\Cont_{\Gammaloop_{g-m,m,[n]}}\mathcal{D}_{g,n}^g(0,\dots,0)   \\
=&\frac{1}{m!\cdot 2^m }(-1)^m\sum_{k_1+\dots +k_{m}=g}\left(\prod_{i=1}^{m}\frac{B_{2k_i}}{k_i!}\right) \sum_{b_1=0}^{k_1-1}\dots\sum_{b_m=0}^{k_m-1}(\xi_{\Gammaloop_{g-m,m,[n]}})_*\prod_{i=1}^{m}\binom{k_i-1}{b_i}(\psi_{h_i})^{b_i}(\psi_{h'_i})^{k_i-1-b_i} \\=&\frac{1}{m!\cdot 2^m }(-1)^m\sum_{k_1+\dots +k_{m}=g}\left(\prod_{i=1}^{m}\frac{B_{2k_i}}{k_i!}\right) \sum_{b_1=0}^{k_1-1}\dots\sum_{b_m=0}^{k_m-1}8^{-(g-m)}\prod_{i=1}^{m}\frac{2^{k_i-1}(k_i-1)!}{(2k_i)!}
\binom{2k_i}{2b_i+1}
\left(\xi_{\Gammaloop_{0,g,[n]}}\right)_*(1)
+\cdots
\\=&\frac{1}{m! }(-1)^m\sum_{k_1+\dots +k_{m}=g}\left(\prod_{i=1}^{m}\frac{B_{2k_i}}{k_i\cdot(2k_i)!}\right) \left(\xi_{\Gammaloop_{0,g,[n]}}\right)_*(1)+\cdots.
\end{align*}
The omitted terms in the above formula satisfy the following condition: there are no
$\kappa$ classes and no genus-$h$ vertices with monomial of $\psi$ classes of degree $\geq h+\delta_{h}^0$ for
$0\leq h\leq g$.
Notice that
\begin{align*}
    \sum_{k_1+\dots+k_m=g}\left(\prod_{i=1}^{m}\frac{B_{2k_i}}{k_i\cdot(2k_i)!}\right) =\left[\left(\sum_{j=1}^{\infty}\frac{B_{2j}}{(2j)!}\frac{1}{j}t^{2j}\right)^{m}\right]_{t^{2g}}
    =\left[\left(-t+2\ln\frac{e^t-1}{t}\right)^m\right]_{t^{2g}}.
\end{align*}
In fact,  by definition of Bernoulli number, 
\begin{align*}
   \int t^{-1}\left(\frac{t}{2}\frac{e^t+1}{e^t-1}-1 \right)dt=\sum_{j=1}^{\infty}\frac{B_{2j}}{(2j)!}\frac{1}{2j}t^{2j}=-\frac{t}{2}+\ln\frac{e^t-1}{t}.
\end{align*}
Thus
the coefficient of the bouquet class $\left(\xi_{\Gammaloop_{0,g,[n]}}\right)_*(1)$ in $\lambda_g$ equals 
\begin{align*}
&\frac{(-1)^g}{2^g}\sum_{m=1}^{g}\frac{(-1)^m}{m!}\sum_{k_1+\dots+k_m=g}\prod_{i=1}^{m}\frac{B_{2k_i}}{k_i\cdot (2k_i)!}=\left[\frac{(-1)^g}{2^g}\sum_{m=1}^{g}\frac{(-1)^m}{m!}\left(-t+2\ln\frac{e^t-1}{t}\right)^{m}\right]_{t^{2g}}
\\=&\left[\frac{(-1)^g}{2^g}\sum_{m=1}^{\infty}\frac{(-1)^m}{m!}\left(-t+2\ln\frac{e^t-1}{t}\right)^{m}\right]_{t^{2g}}
\\=&\left[\frac{(-1)^g}{2^g}\frac{t^2e^t}{(e^t-1)^2}\right]_{t^{2g}}.
\end{align*}

\end{proof}

\section{Intersection numbers on the moduli space of stable curves}
\label{sec:inter-number-mgn}

In this section, we modify the algorithm from Section~\ref{subsec:algorithm-trr} to derive a recursive algorithm for the intersection numbers $\int_{\overline{\mathcal{M}}_{g,n}}\prod_{i=1}^{n}{\psi_i}^{k_i}$ with any $2g-2+n>0$.
We will show that all intersection numbers are uniquely determined by the recursion with initial conditions $\int_{\overline{\mathcal{M}}_{0,3}}1=1$ and $\int_{\overline{\mathcal{M}}_{1,1}}\psi_1=\frac{1}{24}$.

Fix a genus $g$
and a number of marked points $n$ satisfying $2g-3+n>0$. Let $M$ be a monomial of degree
$D\leq 2(3g-3+n)$ in the variables $a_2,\dots,a_n$, and let
$N:=n+2(3g-3+n)-D$. Define
\[\widehat{\Omega}^{pre}_{g,M}=[\mathcal{D}_{g,N}^{3g-3+n}(a_2,\dots,a_N)]_{M\cdot a_{n+1}\dots a_{N}}\in \mathcal{S}_{g,N}^{3g-3+n}\]
to be the coefficient of the monomial $M\cdot a_{n+1}\dots a_N$ in Pixton’s class $\mathcal{D}_{g,N}^{3g-3+n}(a_2,\dots,a_N)$. 
Let $\pi\colon \overline{\mathcal{M}}_{g,N}\rightarrow \overline{\mathcal{M}}_{g,n}$ be the forgetful map. Define 
\[\widehat{\Omega}_{g,M}=\pi_*\left(\widehat{\Omega}_{g,M}^{pre}\cdot \psi_{n+1}\dots\psi_{N}\right)\in \mathcal{S}^{3g-3+n}_{g,n}.\]
In the range $2g-3+n>0$, integrating the class $\widehat{\Omega}_{g,M}$ gives
\begin{align}\label{eqn:int-trr-3g-3+n}
\int_{\overline{\mathcal{M}}_{g,n}}\widehat{\Omega}_{g,M} =0. 
\end{align}
The following Theorem~\ref{thm:algo-inter} shows that \eqref{eqn:int-trr-3g-3+n}
yields a new combinatorial recursive formula for integrals
of $\psi$ classes on the moduli space of curves. 

To state the theorem, first, fix integers $(a_2, \dotsc, a_n)$, an integer $r > 0$, a
stable graph $\Gamma=(V,H,g,p,\iota)$ of genus $g$ with $n$ legs, and an  assignment $x_v \in \mathbb Z$ for
every $v \in V$.
Given this, we set
\begin{equation*}
  a_1 = -(a_2 + \dotsb + a_n + \sum_{v \in V} x_v).
\end{equation*}
With this, we define a \emph{weighting modulo $r$} on $(\Gamma, x)$ to be a
map
\begin{equation*}
  w\colon H \to \{0, \dotsc, r - 1\}
\end{equation*}
satisfying three properties:
\begin{enumerate}
\item For any $i \in \{1, \dotsc, n\}$ corresponding to a leg $\ell_i$
  of $\Gamma$, we have $w(\ell_i) \equiv a_i \pmod{r}$.
\item For any edge $e\in E$ corresponding to two half-edges $h, h'\in H$, we
  have $w(h) + w(h') \equiv 0 \pmod{r}$.
\item For any vertex $v \in V$, we have
  $\sum_{h\in p^{-1}(v)} w(h) \equiv -x_v \pmod{r}$.
\end{enumerate}
Define $\widehat{\mathcal{D}}_{(\Gamma,x)}^r$ to be the class
\begin{equation*}
  \frac 1{r^{h^1(\Gamma)}} \sum_{\substack{w \text{ weighting }\\ \text{mod } r \text{ on }(\Gamma,x)}} \left[\Gamma, \prod_{i=1}^n e^{\frac 12 a_i^2\psi_i} \prod_{(h,h')\in E} \frac{1 - e^{-\frac 12 w(h)w(h')(\psi_h + \psi_{h'})}}{\psi_h + \psi_{h'}}\right] \in \mathcal{S}_{g,n}.
\end{equation*}
This class is a polynomial for $r \gg 0$, and we will use
$\widehat{\mathcal{D}}_{(\Gamma,x)}$ to denote the constant part of the resulting
polynomial.
The class $\widehat{\mathcal{D}}_{(\Gamma,x)}$ is a polynomial in the variables
$a_2, \dotsc, a_n$ and $x_v$ for $v \in V$, and so we will
denote it by
\begin{equation*}
\widehat{\mathcal{D}}_{\Gamma}(a_2, \dotsc, a_n, \{x_v\}).
\end{equation*}

  

\begin{theorem}\label{thm:algo-inter}
  For any $\sum_{i=1}^{n}k_i=3g-3+n$, we have
\begin{align}\label{eqn:deg-3g-3+n-push-algo-int}
\int_{\overline{\mathcal{M}}_{g,n}}  \prod_{i=1}^{n}{\psi_i}^{k_i}   =\sum_{\substack{\Gamma\in G_{g,n}\\ |E(\Gamma)|\geq1}}\sum_{\{\alpha_{v}:v\in V(\Gamma)\}} c_{(\Gamma,\{\alpha_v\}_{v\in V(\Gamma)})}\cdot \prod_{v\in V(\Gamma)}\int_{\overline{\mathcal{M}}_{g(v),n(v)}}\alpha_{v}  
\end{align}
where  $\alpha_v$ ranges over all possible  monomials of tautological $\psi$ classes on $\overline{\mathcal{M}}_{g(v),n(v)}$.
Each coefficient $c_{(\Gamma,\{\alpha_v\}_{v\in V(\Gamma)})}$ is a rational number coming from the weight on stable graphs in Pixton's double ramifcation formula and factors from string equation and dilaton equation.        
\end{theorem}
\begin{proof}
The proof is similar to the proof of Proposition \ref{prop:omega_g,m-formula}.   First, the multiplication by $\psi_{n+1} \cdot \dots \cdot \psi_N$
  in the definition of $\widehat{\Omega}_{g, M}$ ensures that we only need to
  consider $N$-pointed stable graphs obtained by distributing
  $N-n = 2(3g-3+n)-D$ legs to an genus-$g$, $n$-pointed stable graph
  $\Gamma$.
  Let $m(v)$ be the number of extra legs at a vertex $v$ of $\Gamma$.
  Then, there are
  \begin{equation*}
    \frac{(2(3g-3+n)-D)!}{\prod_v m(v)!}
  \end{equation*}
  many such distributions.
  Because of the identity
  \begin{equation*}
    [p(x_1 + \dotsb + x_m)]_{x_1 \cdot \dotsb \cdot x_m}
    = m! [p(x)]_{x^m}
  \end{equation*}
  for any polynomial $p(x)$, we may set $x_v$ for $v$ be the
  sum of the variables $a_i$ for $i$ corresponding to extra markings
  at $v$.
  This leads to the additional factors
  \begin{equation*}
\prod_{v} m(v)!.  \end{equation*}
  Finally, applying the dilaton equation repeatedly to compute the
  push-forward under the map forgetting the last $2(3g-3+n)-D$ markings
  leads to the factor
  \begin{equation*}
    \frac{(2g(v) - 3 + n(v) + m(v))!}{(2g(v) - 3 + n(v))!}
  \end{equation*}
  for every vertex $v$.
  Putting all of these together yields the formula
  \begin{align}\label{eqn:hat-omega_g,m-formula}
   \widehat{\Omega}_{g, M}
    = \sum_{\Gamma\in G_{g,n}} \sum_{\substack{m\colon V \to \mathbb Z_{\ge 0} \\ \sum_v m(v) = 2(3g-3+n)-D}}
    \frac{(2(3g-3+n)-D)!}{|\Aut(\Gamma)|}
    \prod_{v \in V} \frac{(2g(v) - 3 + n(v) + m(v))!}{(2g(v) - 3 + n(v))!} \cdot \widehat{C}_{\Gamma, M, m},
  \end{align}
  where 
  $\widehat{C}_{\Gamma, M, m}$ is the coefficient of
  $M \cdot \prod_v x_v^{m(v)}$ of the degree $3g-3+n$ part of
  the class $\widehat{\mathcal{D}}_{\Gamma}(a_2, \dotsc, a_n, \{x_v\})$.

Combining equations~\eqref{eqn:int-trr-3g-3+n} and \eqref{eqn:hat-omega_g,m-formula}, we obtain
\begin{align}\label{eqn:graph-sum-deg-3g-3+n-int}
\sum_{\Gamma\in G_{g,n}} \sum_{\substack{m\colon V \to \mathbb Z_{\ge 0} \\ \sum_v m(v) = 2(3g-3+n)-D}}
    \frac{(2(3g-3+n)-D)!}{|\Aut(\Gamma)|}
    \prod_{v \in V} \frac{(2g(v) - 3 + n(v) + m(v))!}{(2g(v) - 3 + n(v))!} \cdot \int_{\overline{\mathcal{M}}_{g,n}}\widehat{C}_{\Gamma, M, m}=0.
  \end{align}
  Finally, the above equation~\eqref{eqn:graph-sum-deg-3g-3+n-int} leads to equation~\eqref{eqn:deg-3g-3+n-push-algo-int}: letting $M=\prod_{i=2}^{n}a_i^{2k_i}$,
  the contribution of the trivial stable graph $\Gamma_{g,[n]}$ in equation~\eqref{eqn:graph-sum-deg-3g-3+n-int} becomes
  \begin{align}\label{eqn:contr-trivial-inter-number}
  &(6g-6+2n-D)!\frac{(2g-3+n+2(3g-3+n)-D)!}{(2g-3+n)!} 
\nonumber\\&\hspace{10pt}\cdot\int_{\overline{\mathcal{M}}_{g,n}}\left[\widehat{\mathcal{D}}_{\Gamma_{g,[n]}}^{3g-3+n}(a_2,\dots,a_n,\{x\})\right]_{\prod_{i=2}^{n}a_i^{2k_i}\cdot x^{6g-6+2n-D}}
\nonumber\\=&(6g-6+2n-D)!\frac{(8g-9+3n-D)!}{(2g-3+n)!} 
\int_{\overline{\mathcal{M}}_{g,n}}\left[e^{\frac{1}{2}(x+\sum_{i=2}^{n}a_i)^2\psi_1}\prod_{i=2}^{n}e^{\frac{1}{2}a_i^2\psi_i}\right]_{\prod_{i=2}^{n}a_i^{2k_i}\cdot x^{6g-6+2n-D}}
\nonumber\\=&(6g-6+2n-D)!\frac{(8g-9+3n-D)!}{(2g-3+n)!} 
\nonumber\\&\cdot
\sum_{\sum_{i=1}^{n}l_i=3g-3+n}\prod_{i=1}^{n}\frac{1}{2^{l_i}l_i!}
\int_{\overline{\mathcal{M}}_{g,n}}\prod_{i=1}^{n}\psi_i^{l_i}\left[(x+\sum_{i=2}^{n}a_i)^{2l_1}\right]_{\prod_{i=2}^{n}a_i^{2k_i-2l_i}\cdot x^{6g-6+2n-D}}
\nonumber\\=&(6g-6+2n-D)!\frac{(8g-9+3n-D)!}{(2g-3+n)!} 
\nonumber\\&\cdot
\sum_{\sum_{i=1}^{n}l_i=3g-3+n}\prod_{i=1}^{n}\frac{1}{2^{l_i}l_i!}
\int_{\overline{\mathcal{M}}_{g,n}}\prod_{i=1}^{n}\psi_i^{l_i}
\binom{2l_1}{2k_2-2l_2,\dots,2k_n-2l_n,6g-6+2n-2\sum_{i=2}^{n}k_i}.
  \end{align}
The contribution of all stable graphs with at least one edge in equation~\eqref{eqn:graph-sum-deg-3g-3+n-int} is equal to 
\begin{align}\label{eqn:contr-other-graph-inter-number}
&\sum_{\Gamma\in G_{g,n}: |E(\Gamma)|\geq1}\sum_{\substack{m:V\rightarrow\mathbb{Z}_{\geq0}\\
\sum_{v}m(v)=2(3g-3+n)-D}}\frac{(6g-6+2n-D)!}{|\Aut(\Gamma)|}
\nonumber\\&\hspace{20pt}\cdot\prod_{v\in V(\Gamma)}\frac{(2g(v)-3+n(v)+m(v))!}{(2g(v)-3+n(v))!}\int_{\prod_{v}\overline{\mathcal{M}}_{g(v),n(v)}}\left[\widehat{\mathcal{D}}_{\Gamma}^{3g-3+n}(a_2,\dots,a_n,\{x_v\})\right]_{\prod_{i=2}^{n}a_i^{2k_i}\cdot\prod_{v\in V(\Gamma)}x_v^{m_v}}.    \end{align}
Plugging equations~\eqref{eqn:contr-trivial-inter-number} and \eqref{eqn:contr-other-graph-inter-number} into equation~\eqref{eqn:graph-sum-deg-3g-3+n-int}, we obtain
\begin{align*}
\int_{\overline{\mathcal{M}}_{g,n}}\prod_{i=1}^{n}\psi_i^{k_i} +\sum_{(l_2,\dots,l_n)<(k_2,\dots,k_n)}c'_{l_2,l_3,\dots,l_n}
\int_{\overline{\mathcal{M}}_{g,n}}\prod_{i=1}^{n}\psi_i^{l_i} =\text{boundary terms}
\end{align*}
where $(l_2,\dots,l_n)<(k_2,\dots,k_n)$ means the monomials of $\psi$ classes $\prod_{j=1}^{n}\psi_j^{l_j}$ strictly lower than  $\prod_{j=1}^{n}\psi_j^{k_j}$. Here all $c'_{k_2,k_3,\dots,k_n}$ are rational numbers determined from equation~\eqref{eqn:contr-trivial-inter-number}.
Finally, by induction on $(k_2,\dots,k_n)$, we obtain equation~\eqref{eqn:deg-3g-3+n-push-algo-int}. 
\end{proof}

Now we give an  example to show how to use Theorem~\ref{thm:algo-inter} to do explicit calculations.
\begin{example}\label{ex:m12}
  We verify that
\begin{align}
\label{eqn:ex-M-1-2} \int_{\overline{\mathcal{M}}_{1,2}}\psi_1^2=\frac{1}{24}   
\end{align}
In this case, $(g,n,N,D)=(1,2,6,0)$, and
there are 5 stable graphs in $G_{1,2}$ (see Figure~\ref{fig:graphs}). 
\begin{figure}
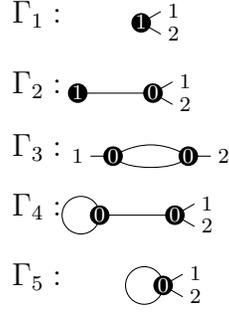

  \centering
  \begin{align*}
    \Gamma_1: & \qquad
                \tikz[baseline]{
                \draw (0,0) -- +(30:3mm) node[label={[label distance=-2mm]0:$\substack{1}$}] {};
                \draw (0,0) -- +(-30:3mm) node[label={[label distance=-2mm]0:$\substack{2}$}] {};
                \fill (0,0) circle(1.3mm) node {\color{white}$\substack 1$};
                } \\
    \Gamma_2: &
                \tikz[baseline]{
                \draw (0,0) -- (1,0);
                \draw (1,0) -- +(30:3mm) node[label={[label distance=-2mm]0:$\substack{1}$}] {};
                \draw (1,0) -- +(-30:3mm) node[label={[label distance=-2mm]0:$\substack{2}$}] {};
                \fill (0,0) circle(1.3mm) node {\color{white}$\substack 1$};
                \fill (1,0) circle(1.3mm) node {\color{white}$\substack 0$};
                } \\
    \Gamma_3: &
                \tikz[baseline]{
                \draw plot [smooth,tension=1] coordinates {(0,0) (0.5,0.15) (1,0)}; \draw plot [smooth,tension=1.5] coordinates {(0,0) (0.5,-0.15) (1,0)};
                \draw (0,0) -- (-0.3,0) node[label={[label distance=-2mm]180:$\substack{1}$}] {};
                \draw (1,0) -- (1.3,0) node[label={[label distance=-2mm]0:$\substack{2}$}] {};
                \fill (0,0) circle(1.3mm) node {\color{white}$\substack 0$}; \fill (1,0) circle(1.3mm) node {\color{white}$\substack 0$};
                } \\
    \Gamma_4: &
                \tikz[baseline]{
                \draw (0,0) -- (1,0);
                \draw (1,0) -- +(30:3mm) node[label={[label distance=-2mm]0:$\substack{1}$}] {};
                \draw (1,0) -- +(-30:3mm) node[label={[label distance=-2mm]0:$\substack{2}$}] {};
                \draw (-0.25,0) circle(0.25);
                \fill (0,0) circle(1.3mm) node {\color{white}$\substack 0$} {};
                \fill (1,0) circle(1.3mm) node {\color{white}$\substack 0$} {};
                } \\
    \Gamma_5: & \qquad
                \tikz[baseline]{
                \draw (0,0) -- +(30:3mm) node[label={[label distance=-2mm]0:$\substack{1}$}] {};
                \draw (0,0) -- +(-30:3mm) node[label={[label distance=-2mm]0:$\substack{2}$}] {};
                \draw (-0.25,0) circle(0.25);
                \fill (0,0) circle(1.3mm) node {\color{white}$\substack{0}$};
                }
  \end{align*}
  \caption{Dual graphs for $\overline{\mathcal M}_{1,2}$}
  \label{fig:graphs}
\end{figure}
Below, we list the  contribution of each stable graph to the equation~\eqref{eqn:int-trr-3g-3+n} or \eqref{eqn:graph-sum-deg-3g-3+n-int}.
Using \eqref{eqn:contr-trivial-inter-number} and \eqref{eqn:contr-other-graph-inter-number} and a direct computation, we see
\begin{align*}
&\Cont_{\Gamma_1}=4!\frac{5!}{1!}\frac{1}{2^2 2!}\binom{4}{0,4}\int_{\overline{\mathcal{M}}_{1,2}}\psi_{1}^2=360\int_{\overline{\mathcal{M}}_{1,2}}\psi_1^2,\,
\\&\Cont_{\Gamma_2}=\frac{4!}{1}\sum_{m_1+m_2=4}\frac{(m_1)!}{(0)!}\frac{(m_2)!}{(0)!}\int_{\overline{\mathcal{M}}_{1,1}\times \overline{\mathcal{M}}_{0,3}}\left[e^{\frac{1}{2}(a_2+x_1+x_2)^2\psi_1}e^{\frac{1}{2}a_2^2\psi_2} \frac{1 - e^{-\frac 12 x_1^2(\psi_h + \psi_{h'})}}{\psi_h + \psi_{h'}}\right]_{x_1^{m_1}x_2^{m_2}}= -3,\,\,
\\&\Cont_{\Gamma_3}=\frac{4!}{2}\sum_{m_1+m_2=4}\frac{m_1!}{0!}\frac{m_2!}{0!}
\Coeff_{r^0}\Bigg[\frac{1}{r}\sum_{\substack{w_1,w_2=0\\ w_1+w_2=a_2+x_2 \mod r}}^{r-1}\int_{\overline{\mathcal{M}}_{0,3}\times\overline{\mathcal{M}}_{0,3}}\frac{1-e^{\frac{1}{2}w_1^2(\psi_{h(e_1)}+\psi_{h'(e_1)})}}{\psi_{h(e_1)}+\psi_{h'(e_1)}}
\\&\hspace{250pt}\cdot\frac{1-e^{\frac{1}{2}w_2^2(\psi_{h(e_2)}+\psi_{h'(e_2)})}}{\psi_{h(e_2)}+\psi_{h'(e_2)}}\Bigg]_{x_1^{m_1}x_2^{m_2}}
\\&\hspace{40pt}= -12,\,  \, 
\\&
\Cont_{\Gamma_4}=\frac{4!}{2}\sum_{m_1+m_2=4}\frac{m_1!}{0!}\frac{m_2!}{0!}
\\&\hspace{50pt}\cdot\Coeff_{r^0}\left[\frac{1}{r}\sum_{w=0}^{r-1}\int_{\overline{\mathcal{M}}_{0,3}\times\overline{\mathcal{M}}_{0,3}}\frac{1-e^{\frac{1}{2}w^2(\psi_{h(e_1)}+\psi_{h'(e_1)})}}{\psi_{h(e_1)}+\psi_{h'(e_1)}}\frac{1-e^{\frac{1}{2}x_1^2(\psi_{h(e_2)}+\psi_{h'(e_2)})}}{\psi_{h(e_2)}+\psi_{h'(e_2)}}\right]_{x_1^{m_1}x_2^{m_2}}
\\&\hspace{40pt}=0,
\\&
\Cont_{\Gamma_5}=\frac{4!}{2}\frac{5!}{1!}\Coeff_{r^0}\left[\frac{1}{r}\sum_{w=0}^{r-1}\int_{\overline{\mathcal{M}}_{0,4}}\left[e^{\frac{1}{2}(a_2+x)^2\psi_1}e^{\frac{1}{2}a_2^2\psi_2}\frac{1-e^{\frac{1}{2}w^2(\psi_{h}+\psi_{h'})}}{\psi_{h}+\psi_{h'}}\right]_{x^4}\right]=0.
\end{align*}
Solving $\sum_{i = 1}^5 \Cont_{\Gamma_i} = 0$ for $\int_{\overline{\mathcal{M}}_{1,2}} \psi_1^2$, we indeed obtain \eqref{eqn:ex-M-1-2}.
\end{example}


\begin{appendix}\label{append}
\section{Ancestor-descendant correspondence}
  
In this appendix, we discuss results about the correspondence between descendant and ancestor Gromov--Witten invariants.
In Subsection~\ref{app:review}, we review standard results, following \cite[Appendix 2]{coates2007quantum}.
In the remaining subsections we discuss the relations between Liu's $T$-operator, Givental's $S$-operator and the ancestor-descendant correspondence.
We expect that many of these results are known to experts, but we were not able to find a reference.

\subsection{Descendants and ancestors}
\label{app:review}

We review some basics of descendant and ancestor correlation functions of Gromov-Witten invariants. 
Consider the stablization morphism
\[St\colon \overline{\mathcal{M}}_{g,n}(X,\beta)\rightarrow \overline{\mathcal{M}}_{g,n}\]
which forgets the map $f$ in $(C; x_1,\dots,x_n; f ) \in  \overline{\mathcal{M}}_{g,n}(X,\beta)$ and stabilizes the curve
$(C; x_1,\dots,x_n)$
by contracting unstable components to points. 
Let
\[\pi_m\colon \overline{\mathcal{M}}_{g,n+m}(X,\beta)\rightarrow \overline{\mathcal{M}}_{g,n}(X,\beta)\]
be the morphism, which forgets the last $m$ markings, and stabilizes the curve.
For $1\leq i\leq n$, denote the  cotangent line bundle  along the $i$-th marked point over $\overline{\mathcal{M}}_{g,n}$ and $\overline{\mathcal{M}}_{g,n+m}(X,\beta)$ by $L_i$ and $\mathcal{L}_i$, respectively.
The bundles $\mathcal{L}_i$ and $\pi_m^* St^*L_i$  over
$\overline{\mathcal{M}}_{g,n+m}(X,\beta)$ are identified outside the locus $D$ consisting of maps such that the $i$-th marked
point is situated on a component of the curve which gets collapsed by $St \circ \pi_m$.

This locus $D$ is the image of the gluing map
\begin{align*}
   \iota\colon \sqcup_{\substack{m_1+m_2=m\\ \beta_1+\beta_2=\beta}} \overline{\mathcal{M}}_{0,\{i\}+\bullet+m_1}(X,\beta_1)\times_{X} \overline{\mathcal{M}}_{g,[n]\setminus\{i\}+\circ+m_2}(X,\beta_2)\rightarrow \overline{\mathcal{M}}_{g,n+m}(X,\beta)
\end{align*}
where the two markings $\bullet$ and $\circ$ are glued together under map $\iota$.
We denote the domain of this map by $Y_{n,m,\beta}^{(i)}$. The virtual normal bundle to $D$ at a generic
point is $\operatorname{Hom}(\pi_m^* St^*c_1(L_i), \mathcal{L}_i)$, and $D$ is ``virtually Poincar\'e-dual'' to $c_1(\mathcal{L}_i)-\pi_m^* St^*c_1(L_i)$ in the sense that
\begin{align}
\label{eqn:D-Y-vir}  \left(c_1(\mathcal{L}_i)-\pi_m^* St^*c_1(L_i)\right) \cap [\overline{\mathcal{M}}_{n+m}{(X,\beta)}]^{vir} =\iota_*[Y_{n,m,\beta}^{(i)}]^{vir}.
\end{align}
Define mixed type of ancestor and descendant correlation function as follows:
\begin{align*}
&\langle\langle\bar{\tau}_{j_1}\tau_{i_1}(\phi_1),\dots,\bar{\tau}_{j_n}\tau_{i_n}(\phi_n)\rangle\rangle_g(\mathbf{t}(\psi))
\nonumber\\&:=\sum_{m,\beta\geq0}\frac{Q^\beta}{m!}
\int_{[\overline{\mathcal{M}}_{g,n+m}(X,\beta)]^{vir}} \prod_{k=1}^{n}\left(\bar{\psi}_{k}^{j_k}\psi_k^{i_k}ev_k^*\phi_k\right)\prod_{k=n+1}^{n+m}ev_k^*\mathbf{t}(\psi_{k}),
\end{align*}
 where $\psi_k:=c_1(\mathcal{L}_k)$, $\bar{\psi}_k:=\pi_m^*St^*c_1(L_k)$ and $\mathbf{t}(z)=\sum_{n,\alpha}t_n^\alpha z^n\phi_\alpha\in H^*(X)[[\{t_n^\alpha\}]][[z]]$. 
\subsection{From ancestor to descendant correlation functions}\label{subsec:anc-des} 
\begin{lemma}\label{lem:anc-desc}
For any $2g-2+n>0$,  on the big phase space, we have
\begin{align}\label{eqn:anc-des}
\langle\langle\bar\tau_{j_1}\tau_{i_1}(\phi_{\alpha_1}),\dots,\bar\tau_{j_n}\tau_{i_n}(\phi_{\alpha_n})\rangle\rangle_g =   \langle\langle  T^{j_1}(\tau_{i_1}(\phi_{\alpha_1})),\dots,  T^{j_n}(\tau_{i_n}(\phi_{\alpha_n}))\rangle\rangle_g 
\end{align}
for any $\{\phi_{\alpha_i}: i=1,\dots,n\}\subset H^*(X;\mathbb{C})$ and $T$ is the operator from \eqref{eqn:T-operator}.
In particular,
\begin{align*}
\langle\langle\bar\tau_{j_1}(\phi_{\alpha_1}),\dots,\bar\tau_{j_n}(\phi_{\alpha_n})\rangle\rangle_g =   \langle\langle  T^{j_1}(\phi_{\alpha_1}),\dots,  T^{j_n}(\phi_{\alpha_n})\rangle\rangle_g. 
\end{align*}
\end{lemma}
\begin{proof}
  We will focus on the first marked point, and for simplicity, suppress the content of the other marked
  points from our notation.
  By equation \eqref{eqn:D-Y-vir}, the following identity holds for mixed type twisted correlators on the big phase space
\begin{align}
\label{eqn:des-anc-recursion}
\langle\langle\bar{\tau}_{j_1-1}\tau_{i_1+1}(\phi),\dots\rangle\rangle_g=
\langle\langle\bar{\tau}_{j_1}\tau_{i_1}(\phi),\dots\rangle\rangle_g
+\langle\langle \tau_{i_1}(\phi),\phi_\beta\rangle\rangle_0\langle\langle\bar{\tau}_{j_1-1}(\phi^\beta),\dots\rangle\rangle_g
\end{align}
for any $\phi\in H^*(X;\mathbb{C})$.

This implies
\begin{align*}
&\langle\langle\bar{\tau}_{j_1}\tau_{i_1}(\phi_{\alpha_1}),\dots\rangle\rangle_g 
=\langle\langle\bar\tau_{j_1-1}\left(\tau_{i_1+1}(\phi_{\alpha_1})-\langle\langle\tau_{i_1}(\phi_{\alpha_1}),\phi_\beta\rangle\rangle_0\phi^\beta\right),\dots\rangle\rangle_g
\\=&\langle\langle\bar\tau_{j_1-1}T(\tau_{i_1}(\phi_{\alpha_1})),\dots\rangle\rangle_g.
\end{align*}
Repeating this $j_1 - 1$ more times, we get
\begin{align*}
  \langle\langle\bar{\tau}_{j_1}\tau_{i_1}(\phi_{\alpha_1}),\dots\rangle\rangle_g 
  =\langle\langle T^{j_1}(\tau_{i_1}(\phi_{\alpha_1})),\dots\rangle\rangle_g
\end{align*}
Similar analysis for the other markings, yields \eqref{eqn:anc-des}.
\end{proof}

For any $g, n$ such that $2g - 2 + n > 0$ and $\mathbf{t}=\sum_{n,\alpha}t_n^\alpha z^n\phi_\alpha\in H^*(X)[[\{t_n^\alpha\}]][[z]]$, define a homomorphism $\Omega^{\mathbf{t}}_{g,n}\colon H^*(X)[[\psi]]^{\otimes n}\rightarrow H^*(\overline{\mathcal{M}}_{g,n})[[Q]][[\{t_n^\alpha\}]]$ via
\begin{align*}
\Omega^{\mathbf{t}}_{g,n}(\phi_{\alpha_1}\psi_1^{i_1},\dots,\phi_{\alpha_n}\psi_n^{i_n})  
=\sum_{m,\beta\geq0}\frac{Q^\beta}{m!} St_*{\pi_m}_*\left(\prod_{k=1}^{n}ev_k^*\phi_{\alpha_k}\psi_k^{i_k}\prod_{k=n+1}^{n+m}ev_i^*\mathbf{t}(\psi_{k})\right).
\end{align*}
From the ``cutting edges'' axiom of the virtual cycles of the moduli of stable maps to $X$, this system of homomorphisms satisfies the splitting axioms:
\begin{align}\label{eqn:split-ax-1}
\iota_{g_1,g_2}^*\Omega_{g,n}^{\mathbf{t}}(\phi_{\alpha_1}\psi_1^{i_1},\dots,\phi_{\alpha_n}\psi_{n}^{i_n})
=\sum_{\beta}\Omega_{g_1,n_1+1}^{\mathbf{t}}(\otimes_{k\in I_1}\phi_{\alpha_k}\psi_k^{i_k},\phi_\beta)\otimes \Omega_{g_2,n_2+1}^{\mathbf{t}}(\otimes_{k\in I_2}\phi_{\alpha_k}\psi_k^{i_k}\otimes \phi^\beta)
\end{align}
and
\begin{align}\label{eqn:split-ax-2}
\iota_{g-1}^*\Omega_{g,n}^{\mathbf{t}}(\phi_{\alpha_1}\psi_1^{i_1},\dots,\phi_{\alpha_n}\psi_n^{i_n})
=\sum_{\beta}\Omega^{\mathbf{t}}_{g-1,n+2}(\phi_{\alpha_1}\psi_1^{i_1},\dots,\phi_{\alpha_n}\psi_n^{i_n},\phi_\beta,\phi^\beta)
\end{align}
where $I_1\sqcup I_2=\{1,..,n\}$, $|I_1|=n_1, |I_2|=n_2$ and $\iota_{g_1,g_2}\colon \overline{\mathcal{M}}_{g_1,n_1+1}\times \overline{\mathcal{M}}_{g_2,n_2+1}\rightarrow \overline{\mathcal{M}}_{g,n}$, $\iota_{g-1}\colon \overline{\mathcal{M}}_{g-1,n+2}\rightarrow \overline{\mathcal{M}}_{g,n}$ are the two canonical gluing maps.
Note that, by definition, $\Omega$ is related to the double bracket of the descendant potential via integration on the moduli space of curve, 
i.e.
\begin{align*}
\int_{\overline{\mathcal{M}}_{g,n}}\Omega_{g,n}^{\mathbf{t}}(v_1\psi_1^{i_1},\dots,v_n\psi_{n}^{i_n})
=\langle\langle\tau_{i_1}(v_1),\dots,\tau_{i_n}(v_n)\rangle\rangle_g(\mathbf{t}).
\end{align*} 
For any topological recursion relations $TRR=0$ on $\overline{\mathcal{M}}_{g,n}$ and any vectors $\{v_i\}_{i=1}^{n}\subset H^*(X;\mathbb{Q})$, we obtain the vanishing
\begin{align*}
\int_{\overline{\mathcal{M}}_{g,n}}\Omega^{\mathbf{t}}(v_1,\dots,v_n)\cdot TRR=0.    
\end{align*}
Combined with Lemma~\ref{lem:anc-desc},  the splitting axioms~\eqref{eqn:split-ax-1} and \eqref{eqn:split-ax-2} explain the rule of translating topological recursion relations to universal equations for Gromov--Witten invariants.
\subsection{From descendant correlation functions to ancestor correlation functions}\label{subsec:des-anc}
Recall the  operator $S(t,z)\colon H^*(X)\rightarrow H^*(X)[[z^{-1}]]$, defined by Givental (cf.\ \cite{givental2001semisimple})
\begin{align*}
S(t,z)\phi=\phi+\sum_{\alpha}\langle\langle \frac{\phi}{z-\psi},\phi_\alpha\rangle\rangle_0\phi^\alpha.    
\end{align*}

\begin{lemma} For any $2g-2+n>0$, on the big phase space, we have
\begin{align}\label{eqn:des-anc-corre}
\langle\langle\tau_{j_1}(\phi_{\alpha_1}),\dots,\tau_{j_n}(\phi_{\alpha_n})\rangle\rangle_g=\langle\langle\widetilde{\tau_{j_1}(\phi_{\alpha_1})},\dots,\widetilde{\tau_{j_n}(\phi_{\alpha_n}})\rangle\rangle_g  
\end{align}
where $\widetilde{\tau_{j_i}(\phi_{\alpha_i})}=[z^{j_i}\cdot S(t,z)\phi_{\alpha_i}]_{+}|_{z\mapsto \bar{\tau}_{+}}$ and $\bar{\tau}_{+}$ is the operator defined by $\bar{\tau}_{+}(\bar{\tau}_{j}(\phi)):=\bar{\tau}_{j+1}(\phi)$. 
\end{lemma}
\begin{proof}
   By equation~\eqref{eqn:des-anc-recursion}, for any $\phi\in H^*(X;\mathbb{Q})$,  we have
\begin{align*}
\langle\langle\tau_i(\phi),\dots\rangle
\rangle_g
=\langle\langle\bar\tau_1(\tau_{i-1}\phi),\dots\rangle\rangle_g
+\langle\langle\tau_{i-1}(\phi),\phi_\alpha\rangle\rangle_0\langle\langle\phi^\alpha,\dots\rangle\rangle_g
\end{align*}
which,
 applying equation~\eqref{eqn:des-anc-recursion} repeatedly, is
 \begin{align*}
 &\langle\langle(\langle\langle\tau_{i-1}(\phi),\phi^\alpha\rangle\rangle_0 \phi_\alpha),\dots\rangle\rangle_g 
 +\langle\langle\bar\tau_{1}(\langle\langle\tau_{i-2}(\phi),\phi^\alpha\rangle\rangle_0 \phi_\alpha),\dots\rangle\rangle_g
\\&+\langle\langle\bar\tau_2(\langle\langle\tau_{i-3}(\phi),\phi^\alpha\rangle\rangle_0 \phi_\alpha),\dots\rangle\rangle_g 
+\dots+\langle\langle\bar{\tau}_i(\phi),\dots\rangle\rangle_g.
\end{align*}
Note that, by definition of $S$ operator
\begin{align*}
& \langle\langle\tau_{i-1}(\phi),\phi^\alpha\rangle\rangle_0 \phi_\alpha+\langle\langle\tau_{i-2}(\phi),\phi^\alpha\rangle\rangle_0 z(\phi_\alpha)+\langle\langle\tau_{i-3}(\phi),\phi^\alpha\rangle\rangle_0 z^2(\phi_\alpha)+\dots +z^i(\phi)  
\\=&[z^{i}\cdot S(t,z)\phi]_{+}.
\end{align*}
Thus we obtain
\begin{align*}
\langle\langle\tau_{i}(\phi),\dots\rangle\rangle_g
=\langle\langle [z^{i}\cdot S(t,z)\phi]_{+}|_{z\mapsto \bar{\tau}_{+}},\dots\rangle\rangle_g.
\end{align*}
Applying the same argument at each marked point will give the equation~\eqref{eqn:des-anc-corre}.
\end{proof}

\subsection{Relations between Liu's $T$ operator and Givental's $S$ operator}

From subsections \ref{subsec:anc-des} and \ref{subsec:des-anc}, we can see $T$ operator appears naturally in the transformation from ancestor correlation functions to 
descendant correlation functions, while $S$ operator appears naturally in the transformation from  descendant correlation functions to 
ancestor correlation functions. Roughly speaking, the two operators are inverse to each other.
Precisely, the results in subsections \ref{subsec:anc-des} and \ref{subsec:des-anc} imply the following:
\begin{corollary}
For any $\phi\in H^*(X;\mathbb{Q})$,
\begin{align*}
\left[ S(\mathbf{t},z)(z^n\phi)\right]_{+}\bigg|_{z\rightarrow T}=\tau_n(\phi)    
\end{align*}
where $z\rightarrow T$ means replacing the variable $z$ by the operator $T$. For example,  under $z\rightarrow T$, $z^n\phi$ becomes $T^n(\phi)$.
Furthermore, we have
\begin{align*}
\left[S(\mathbf{t},z)\left(T^{n}(\phi)\big|_{\tau_+\rightarrow z}\right)\right]_{+}=z^n\phi    
\end{align*}
where $\tau_+\rightarrow z$ means replacing the operator $\tau_+$ by multiplying the variable $z$. For example, under $\tau_+\rightarrow z$, the insertion $\tau_{n}(\phi)$ becomes $z^n\phi$.
\end{corollary}

\end{appendix}


\end{document}